\documentclass[11pt, reqno]{amsart}
\usepackage{amsmath, amssymb}
\usepackage{graphicx}
\usepackage{tikz}
\usepackage[all]{xy}
\usepackage{color}

\usetikzlibrary{calc, intersections, positioning}

\newtheorem{theorem}{Theorem}[section]
\newtheorem{lemma}[theorem]{Lemma}
\newtheorem{proposition}[theorem]{Proposition}
\newtheorem{corollary}[theorem]{Corollary}
\theoremstyle{definition}
\newtheorem{definition}[theorem]{Definition}
\newtheorem{example}[theorem]{Example}

\theoremstyle{remark}
\newtheorem{remark}[theorem]{Remark}
\numberwithin{equation}{section}

\newcommand{\Z}{\mathbb{Z}}
\newcommand{\Q}{\mathbb{Q}}
\newcommand{\R}{\mathcal{R}}
\newcommand{\C}{\mathcal{C}}
\newcommand{\D}{\mathfrak{D}}
\newcommand{\I}{\mathcal{I}}
\newcommand{\J}{\mathcal{J}}
\newcommand{\M}{\mathcal{M}}

\newcommand{\g}{\mathfrak{g}}
\renewcommand{\ss}{\mathfrak{s}}

\newcommand{\A}{\mathcal{A}}
\renewcommand{\H}{\mathcal{H}}

\newcommand{\Ztilde}{\widetilde{Z}}
\newcommand{\N}{\overrightarrow{N}}
\newcommand{\Fix}{\operatorname{Fix}}

\newcommand{\rank}{\operatorname{rank}}
\renewcommand{\Im}{\operatorname{Im}}
\newcommand{\id}{\mathrm{id}}
\newcommand{\Ker}{\operatorname{Ker}}

\newcommand{\Aut}{\operatorname{Aut}}

\newcommand{\sgn}{\operatorname{sgn}}
\newcommand{\Int}{\operatorname{Int}}

\newcommand{\rev}{\mathrm{rev}}
\newcommand{\lss}{\tilde{\ss}}

\newcommand{\pr}{\mathrm{pr}}

\newcommand{\tor}{\operatorname{tor}}

\newcommand{\AS}{\mathrm{AS}}

\newcommand{\ideg}{\operatorname{i-deg}}

\newcommand{\bu}{\mathfrak{bu}}
\newcommand{\bd}{\mathfrak{bd}}
\newcommand{\mh}{\mathfrak{mh}}
\newcommand{\mht}{\mathfrak{mht}}
\newcommand{\Id}{\mathrm{Id}}

\newcommand{\ang}[1]{\langle#1\rangle}
\newcommand{\Ang}[1]{\left\langle#1\right\rangle}

\newcommand{\Ygraph}[4]{%
  \begin{tikzpicture}[scale=#1, baseline={(0,-0.1)}, densely dashed]
    \coordinate (origin) at (0,0);
    \draw (origin) -- (-1,1) node[at end, anchor=south east] {$#2$};
    \draw (origin) -- (1,1) node[at end, anchor=south west] {$#3$};
    \draw (origin) -- (0,-1) node[at end, anchor=north] {#4};
  \end{tikzpicture}%
}

\newcommand{\buY}[4]{%
  \begin{tikzpicture}[scale=#1, baseline={(0,-0.3)}, densely dashed]
    \coordinate (origin) at (0,0);
    \draw (origin) -- (150:2) node[at end, anchor=south east] {$#2$} ;
    \draw (origin) -- (30:2) node[at end, anchor=south west] {$#3$};
    \draw (origin) -- (270:2) node[at end, anchor=north] {$#4$};
    \draw[fill=white] (origin) circle [radius=0.8];
  \end{tikzpicture}%
}

\newcommand{\Hgraph}[4]{%
\begin{tikzpicture}[scale=0.5, baseline={(0,-0.1)}, densely dashed]
  \draw (-1,1) -- (-1,-1) node[at start, anchor=south] {$#1$} node[at end, anchor=north] {$#3$};
  \draw (1,1) -- (1,-1) node[at start, anchor=south] {$#2$} node[at end, anchor=north] {$#4$};
  \draw (-1,0) -- (1,0) ;
\end{tikzpicture}%
}

\newcommand{\Igraph}[4]{%
\begin{tikzpicture}[scale=0.5, baseline={(0,-0.1)}, densely dashed]
  \draw (-1,1) -- (1,1) node[at start, anchor=south] {$#1$} node[at end, anchor=south] {$#2$};
  \draw (-1,-1) -- (1,-1) node[at start, anchor=north] {$#3$} node[at end, anchor=north] {$#4$};
  \draw (0,1) -- (0,-1) ;
\end{tikzpicture}%
}

\newcommand{\twist}{%
\begin{tikzpicture}[scale=0.1, baseline={(0,-0.1)}]
 \draw (45:0.5) .. controls +(1,1) and +(0,1) .. (2,0);
 \draw (2,0) .. controls +(0,-1) and +(1,-1) .. (-45:0.5);
 \draw (-45:0.5) -- (135:0.5);
 \draw (135:0.5) .. controls +(-1,1) and +(0,1) .. (-2,0);
 \draw (-2,0) .. controls +(0,-1) and +(-1,-1) .. (-135:0.5);
\end{tikzpicture}%
}

\newcommand{\vtwist}{\rotatebox{90}{\hspace{-0.35em}\twist}}

\makeatletter
\@namedef{subjclassname@2020}{\textup{2020} Mathematics Subject Classification}
\makeatother

\begin{document}

\title[On the kernel of the surgery map]{On the kernel of the surgery map restricted to the 1-loop part}

\author{Yuta Nozaki}
\address{
Graduate School of Advanced Science and Engineering, Hiroshima University \\
1-3-1 Kagamiyama, Higashi-Hiroshima City, Hiroshima, 739-8526 \\
Japan}
\email{nozakiy@hiroshima-u.ac.jp}

\author{Masatoshi Sato}
\address{
Department of Mathematics \\
Tokyo Denki University \\
5 Senjuasahi-cho, Adachi-ku, Tokyo 120-8551 \\
Japan}
\email{msato@mail.dendai.ac.jp}

\author{Masaaki Suzuki}
\address{Department of Frontier Media Science, Meiji University \\
4-21-1 Nakano, Nakano-ku, Tokyo, 164-8525 \\
Japan}
\email{mackysuzuki@meiji.ac.jp}


\maketitle

\begin{abstract}
Every homology cylinder is obtained from Jacobi diagrams by clasper surgery.
The surgery map $\ss \colon \A_n^c \to Y_n\I\C_{g,1}/Y_{n+1}$ is surjective for $n\geq 2$, and its kernel is closely related to the symmetry of Jacobi diagrams.
We determine the kernel of $\ss$ restricted to the 1-loop part after taking a certain quotient of the target.
Also, we introduce refined versions of the AS and STU relations among claspers and study the abelian group $Y_n\I\C_{g,1}/Y_{n+2}$ for $n\geq 2$.
\end{abstract}

\setcounter{tocdepth}{1}
\tableofcontents

\section{Introduction}
\label{sec:Intro}
Let $\Sigma_{g,1}$ be a compact oriented surface of genus $g$ with one boundary component.
In the study of the mapping class group $\M_{g,1}$ of $\Sigma_{g,1}$, the Torelli group $\I=\I_{g,1}$ is one of central interest.
Here $\I=\I_{g,1}$ is defined to be the kernel of the natural homomorphism $\M_{g,1} \to \Aut H_1(\Sigma_{g,1};\Z)$.
One way to investigate $\I=\I_{g,1}$ is by considering its lower central series $\{\I(n)\}_{n\geq 1}$ defined by $\I(1)=\I$ and $\I(n+1)=[\I(n),\I]$ inductively.
Goussarov~\cite{Gou99} and Habiro~\cite{Hab00C} initiated the study of the monoid $\I\C=\I\C_{g,1}$ of homology cylinders over $\Sigma_{g,1}$, which can be regarded as a 3-dimensional analogue of $\I$.
Here a homology cylinder is roughly a 3-manifold that is homologically the product of $\Sigma_{g,1}$ and the closed interval $[-1,1]$.
In their papers, clasper calculus in 3-manifolds was introduced.
A clasper is a certain surface embedded in a 3-manifold, and one obtains another 3-manifold by surgery along the clasper.
Using claspers, they introduced the $Y_n$-equivalence among 3-manifolds for any positive integer $n$ and the $Y$-filtration $\{Y_n\I\C\}_{n\geq 1}$ on $\I\C$ corresponding to $\{\I(n)\}_{n\geq 1}$.
More precisely, we have an injective monoid homomorphism $\I \hookrightarrow \I\C$ which maps $\I(n)$ into $Y_n\I\C$.
Furthermore, the map induces a homomorphism $\bigoplus_{n\geq 1} \I(n)/\I(n+1) \to \bigoplus_{n\geq 1} Y_n\I\C/Y_{n+1}$ between abelian groups.
Thus, determining $Y_n\I\C/Y_{n+1}$ is significant to understand $\I(n)/\I(n+1)$ well.
In fact, the authors~\cite{NSS21} introduced an invariant $\bar{z}_{n+1}\colon Y_n\I\C/Y_{n+1} \to \A_{n+1}^c\otimes\Q/\Z$ and extracted new information about $\I(n)/\I(n+1)$ which is not detected by the Johnson homomorphisms (\cite[Theorem~1.2 and Corollary~1.4]{NSS21}).
Furthermore, by combining \cite[Theorem~1.2]{NSS21} with \cite[Theorem~1.2]{MSS20}, we conclude that the torsion subgroup $\tor(\I(n)/\I(n+1))$ is non-trivial when $n=3,5$ and $g$ is in a stable range.
This is a generalization of the fact that $\tor(\I(1)/\I(2)) \neq 0$.
For more details of the groups $\I(n)/\I(n+1)$ and $Y_n\I\C/Y_{n+1}$ we refer the reader to \cite[Section~1]{NSS21}.

Let $\A_n^c$ denote the $\Z$-module generated by connected Jacobi diagrams with $n$ trivalent vertices subject to the AS, IHX, and self-loop relations (to be defined later).
Here a Jacobi diagram is a uni-trivalent graph with additional information.
We can define the surgery map $\ss \colon \A_n^c \to Y_n\I\C/Y_{n+1}$ by realizing a Jacobi diagram $J$ as a clasper in the trivial homology cylinder $\Sigma_{g,1}\times[-1,1]$, and $\ss$ is surjective (\cite[Section~8.5]{Hab00C}) when $n\geq 2$.
Hence, we focus on the kernel of $\ss$.
Since $\ss\otimes\id_\Q$ is known to be an isomorphism (\cite[Theorem~7.11]{CHM08}), we conclude $\Ker\ss \subset \tor\A_n^c$.
To investigate $\tor\A_n^c$, we decompose $\A_n^c$ into the direct sum $\bigoplus_{l\geq 0}\A_{n,l}^c$, where $\A_{n,l}^c$ denotes the submodule generated by Jacobi diagrams whose first Betti numbers are $l$.
We write $\ss_{n,l}$ for the restriction of $\ss$ to the $l$-loop part $\A_{n,l}^c$.
The 0-loop part $\A_{n,0}^c$ was deeply studied by Conant, Schneiderman, and Teichner~\cite{CST12L, CST12W, CST16}, and it was proved that $\tor\A_{n,0}^c$ is generated by symmetric Jacobi diagrams of a particular form when $n$ is odd.
They applied their results to the study of the homology cobordism group $\I\H=\I\H_{g,1}$ of homology cylinders, which is the group obtained from $\I\C$ by identifying homology cylinders that are homology cobordant.
In the study of $\I\H$, one can ignore claspers having loops, namely the composite map $\A_{n,l}^c \xrightarrow{\ss_{n,l}} Y_n\I\C/Y_{n+1} \twoheadrightarrow Y_n\I\H/Y_{n+1}$ is trivial if $l\geq 1$ due to Levine~\cite{Lev01}, where $Y_n\I\H$ denotes the image of $Y_n\I\C$.
On the other hand, $\ss_{n,l}$ itself is non-trivial in general, and it is hard to determine $\Ker\ss_{n,l}$.

\begin{figure}[h]
 \centering
\begin{tikzpicture}[scale=0.3, baseline={(0,0)}, densely dashed]
 \draw (0,0) circle [radius=2];
 \draw (160:2) -- (160:3) node[anchor=east] {$a_1$};
 \draw (-160:2) -- (-160:3) node[anchor=east] {$a_1$};
 \draw (40:2) -- (40:3) node[anchor=west] {$a_{m-1}$};
 \draw (-40:2) -- (-40:3) node[anchor=west] {$a_{m-1}$};
 \draw (0:2) -- (0:3) node[anchor=west] {$a_{m}$};
 \node at (0,3) {$\cdots$};
 \node at (0,-3) {$\cdots$};
\end{tikzpicture}
 \quad
\begin{tikzpicture}[scale=0.3, baseline={(0,0)}, densely dashed]
 \draw (0,0) circle [radius=2];
 \draw (180:2) -- (180:3) node[anchor=east] {$a_1$};
 \draw (140:2) -- (140:3) node[anchor=east] {$a_2$};
 \draw (-140:2) -- (-140:3) node[anchor=east] {$a_2$};
 \draw (20:2) -- (20:3) node[anchor=west] {$a_{m}$};
 \draw (-20:2) -- (-20:3) node[anchor=west] {$a_{m}$};
 \node at (0,3) {$\cdots$};
 \node at (0,-3) {$\cdots$};
\end{tikzpicture}
 \caption{A pair of two symmetric 1-loop Jacobi diagrams.}
 \label{fig:sym_1loop}
\end{figure}

In this paper, we shall focus on the case $l=1$ and determine $\Ker\ss_{n,1}$ after taking a certain quotient of $Y_n\I\C/Y_{n+1}$.
First, we have to know the structure of the module $\tor\A_{n,1}^c$, which was done in \cite{NSS21}.
When $n$ is odd, $\tor\A_{n,1}^c$ is generated by symmetric 1-loop Jacobi diagrams illustrated in Figure~\ref{fig:sym_1loop}, and we prove that certain pairs of such diagrams are in the kernel.
To mention a rigorous statement, we introduce some modules and a quotient map.
For $n \geq 2$, let $\ang{\Theta_n^{\geq 1}}$ denote the submodule of $\A_{n,2}^{c}$ generated by 2-loop Jacobi diagrams of the form
\newcommand{\thetapqr}{
\begin{tikzpicture}[scale=0.3, baseline={(0,-0.2)}, densely dashed]
 \draw (0,0) circle [radius=3];
 \draw (-2.8,-1) -- (2.8,-1);
 \draw (135:3) -- (135:4) node[anchor=south] {$a_1$};
 \node at (0,4) {$\cdots$};
 \draw (45:3) -- (45:4) node[anchor=south] {$a_p$};
 \draw (-1.5,-1) -- (-1.5,0) node[anchor=south] {$b_1$};
 \node at (0,0) {$\cdots$};
 \draw (1.5,-1) -- (1.5,0) node[anchor=south] {$b_q$};
 \draw (-135:3) -- (-135:4) node[anchor=north] {$c_1$};
 \node at (0,-4) {$\cdots$};
 \draw (-45:3) -- (-45:4) node[anchor=north] {$c_r$};
\end{tikzpicture}
}
\begin{align}
\label{eq:theta_intro}
\thetapqr\ ,
\end{align}
where $p,q,r \geq 1$ satisfy $p+q+r+2=n$, and $a_i, b_i, c_i \in \{1^{+},\dots,g^{+},1^{-},\dots,g^{-}\}$.
Similarly, we write $\ang{\Theta_n^{\geq 1, s}}$ for the submodule of $\A_{n,2}^{c}$ generated by the above 2-loop Jacobi diagrams with additional condition $a_i=a_{p-i+1}$, $b_i=b_{q-i+1}$, and $c_i=c_{r-i+1}$.
Now, let $\pi$ be the quotient map 
\[
\pi\colon Y_n\I\C/Y_{n+1} \to (Y_n\I\C/Y_{n+1})/\ss(\ang{\Theta_n^{\geq 1,s}}).
\]
Then, in Section~\ref{sec:Ker_OneLoop}, we prove the following main result.

\begin{theorem}
\label{thm:Ker_sn1}
Let $m \geq 2$ be an integer.
Then $\Ker(\pi\circ\ss_{2m-1,1})$ is a free $\Z/2\Z$-module of rank $\frac{1}{2}((2g)^m-(2g)^{\lceil m/2 \rceil})$, generated by the sums of pairs drawn in Figure~\textup{\ref{fig:sym_1loop}}.
\end{theorem}

Theorem~\ref{thm:Ker_sn1} gives an upper bound on the order of the module $\Ker \ss_{2m-1,1}$.
Moreover, since $\Ker(\pi\circ\ss_{2m-1,1})$ coincides with $\Ker \ss_{2m-1,1}$ when $m=2,3$ (see Remark~\ref{rem:TrueKer}), Theorem~\ref{thm:Ker_sn1} determines $\Ker \ss_{3,1}$ and $\Ker \ss_{5,1}$.
The proof is divided into two different arguments.
Roughly speaking, one is the non-triviality of $\pi\circ\ss_{2m-1,1}(J)$, where $J$ is a Jacobi diagram drawn in Figure~\ref{fig:sym_1loop}.
This is shown by the invariant $\bar{z}_{2m}$ in Section~\ref{sec:Ker_OneLoop}.
The other is to check that $\pi\circ\ss_{2m-1,1}(J+J')=0$ for the pair of diagrams $J$, $J'$ in Figure~\ref{fig:sym_1loop} (see Corollary~\ref{cor:1LoopRel}).
To show it, we introduce refined versions of the surgery map and of the AS and STU relations among claspers.
That is, we develop clasper calculus in $Y_n\I\C/Y_{n+2}$ for $n \geq 2$.
This quotient is also an abelian group and fits into an exact sequence
\[
0 \to Y_{n+1}\I\C/Y_{n+2} \to Y_n\I\C/Y_{n+2} \to Y_n\I\C/Y_{n+1} \to 0.
\]
When $n=2$, this sequence splits since $Y_2\I\C/Y_3$ is torsion-free (\cite{MaMe13}), and thus we know the structure of $Y_2\I\C/Y_{4}$ by \cite[Theorem~1.7]{NSS21}.
While $Y_n\I\C/Y_{n+2}$ is important in the study of the $Y$-filtration, its structure is still unknown for $n\geq 3$.
As a consequence of our argument, we obtain the following result in Section~\ref{subsec:App_of_refined}.

\begin{theorem}
\label{thm:Y3C/Y5}
The abelian group $Y_3\I\C/Y_5$ is torsion-free.
\end{theorem}

Theorem~\ref{thm:Y3C/Y5} determines $Y_3\I\C/Y_{5}$ (see also Remark~\ref{rem:Y3C/Y5}).
In particular, the above exact sequence does not split when $n=3$.
Furthermore, we can also apply our framework to the homology cobordism group $\I\H$, and determine $Y_3\I\H/Y_{5}$ in Theorem~\ref{thm:Y3H/Y5}.

\begin{figure}[h]
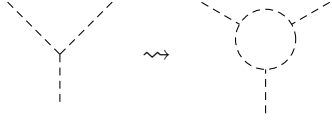

 \centering
 \Ygraph{0.7}{}{}{} $\rightsquigarrow$ \buY{0.5}{}{}{}
 \caption{Process of blowing up a trivalent vertex.}
 \label{fig:bu}
\end{figure}

Additionally, in Section~\ref{sec:Jacobi_diagrams}, we investigate a useful relation between the 1-loop part $\A_{n,1}^c$ and the 2-loop part $\A_{n,2}^c$ to prove Theorem~\ref{thm:Ker_sn1}.
Precisely, we focus on the map $\bu \colon \A_{n,l}^c \to \A_{n+2,l+1}^c$ defined by blowing up a trivalent vertex of a Jacobi diagram as in Figure~\ref{fig:bu}.

\begin{theorem}
\label{thm:bu_isom}
The map $\bu$ induces an isomorphism $\bu \colon \A_{n-2,1}^c \to \A_{n,2}^c/\ang{\Theta_n^{\geq 1}}$ for $n\geq 3$.
\end{theorem}

As an application of Theorem~\ref{thm:bu_isom}, we solve the Goussarov-Habiro conjecture for the $Y_5$-equivalence in Corollary~\ref{cor:GHC}.
This conjecture asserts that the $Y_n$-equivalence among homology cylinders is characterized by finite type invariants of degree at most $n-1$, which is one of the fundamental problems in quantum topology.
The cases $n=2,3$ are known to be true (see \cite[Section~2.3]{MaMe13}), and the case $n=4$ was solved in \cite{NSS21} as a consequence of the determination of the group $Y_3\I\C/Y_{4}$.

Finally, we also discuss higher loop parts $\ss_{n,l}$ in Section~\ref{sec:Ker_HigherLoop}.
In contrast to the 1-loop part, we give lower bounds on the ranks of the $\Z/2\Z$-modules $\Ker(\ss|_{\tor\A_{2k+1,k}^c})$ and $\Im(\ss|_{\tor\A_{2k+1,k}^c})$ for $k\geq 0$.
See Theorem~\ref{thm:HigherLoop} for the precise statement.
We use clasper calculus to study the kernel of $\ss_{2k+1,k}$.
On the other hand, the image of $\ss_{2k+1,k}$ is investigated by the invariant $\bar{z}_{2k+2}$ introduced in \cite{NSS21}.
As a byproduct of the proof of Theorem~\ref{thm:HigherLoop}, we conclude that the $(k+1)$-loop part $\bar{z}_{2k+2,k+1}$ of $\bar{z}_{2k+2}$ is non-trivial in Corollary~\ref{cor:buTaba}.
Here we emphasize that its restriction to the torsion subgroup $\tor(Y_{2k+1}\I\C/Y_{2k+2})$ is still non-trivial and does not factor through the $(2k+1)$-st Johnson homomorphism.
Note that $\bar{z}_{2k,0}$ and $\bar{z}_{2k,1}$ are already known to be non-trivial (\cite[Corollary~1.4]{NSS21}).

\subsection*{Acknowledgments}
This study was supported in part by JSPS KAKENHI Grant Numbers JP20K14317, JP20H05795, JP18K03310, JP20K03596, and JP19H01785.

\section{Preliminaries}
\label{sec:preliminaries}
We review homology cylinders, Jacobi diagrams, and the invariant $\bar{z}_{n+1}$ of homology cylinders.

\subsection{Homology cylinders}
\label{subsec:homology_cylinders}
Let $M$ be a connected oriented compact 3-manifold, and let $m\colon \partial(\Sigma_{g,1}\times[-1,1]) \to \partial M$ be an orientation-preserving homeomorphism.
Two pairs $(M,m)$ and $(M',m')$ are said to be equivalent if there is an orientation-preserving homeomorphism $f\colon M \to M'$ satisfying $f\circ m = m'$.
Let $m_\pm$ denotes the restriction of $m$ to $\Sigma_{g,1}\times\{\pm 1\}$.
A \emph{homology cylinder} over $\Sigma_{g,1}$ is an equivalence class of a pair $(M,m)$ such that $m_{+}$ and $m_{-}$ induce the same isomorphism $H_\ast(\Sigma_{g,1};\Z) \to H_\ast(M;\Z)$.
We define the composition of $M=(M,m)$ and $N=(N,n)$ by stacking $N$ on $M$, that is, $M\circ N = (M\cup_{m_{+}=n_{-}}N, m_{-}\cup n_{+})$.
Then the set $\I\C=\I\C_{g,1}$ of homology cylinders over $\Sigma_{g,1}$ is a monoid.

In the study of the monoid $\I\C$, Goussarov~\cite{Gou99} and Habiro~\cite{Hab00C} introduced claspers.
Roughly speaking, a clasper in a 3-manifold is an embedded surface consisting of disks, bands, and annuli.
Claspers allow us to introduce the $Y_n$-equivalence among homology cylinders and the submonoid $Y_n\I\C$ of $\I\C$.
Also, the descending series $\I\C=Y_1\I\C \supset Y_2\I\C \supset \cdots$ is called the \emph{$Y$-filtration}.
For the definitions of these terminologies, we refer the reader to \cite{Hab00C} or \cite[Sections~2.3 and 2.4]{NSS21}.
It is known that the quotient set $Y_n\I\C/Y_{k}$ is a finitely generated abelian group if $k\leq 2n$.

We sometimes express a homology cylinder by a $2g$-component oriented tangle such as Example~\ref{ex:refined_surgery}.
In fact, there is a bijection between $\I\C$ and the set of bottom-top tangles in a homology cube (\cite[Theorem~2.10]{CHM08}).
The components of a bottom-top tangle correspond to the set $\{1^{+},\dots,g^{+},1^{-},\dots,g^{-}\}$.
Throughout this paper, we fix a symplectic basis $\{\alpha_1,\dots,\alpha_g,\beta_1,\dots,\beta_g\}$ of $H=H_1(\Sigma_{g,1};\Z)$, and identify $H$ with the module $\Z\{1^\pm,\dots,g^\pm\}$ according to $\alpha_i \leftrightarrow i^{-}$ and $\beta_i \leftrightarrow i^{+}$.

\subsection{Jacobi diagrams}
\label{subsec:Jacobi_diagrams}
A \emph{Jacobi diagram} is a uni-trivalent graph such that each trivalent vertex has a cyclic order and each univalent vertex is labeled by an element of the set $\{1^{+},\dots,g^{+},1^{-},\dots,g^{-}\}$.
We define the \emph{internal degree} $\ideg J$ of a Jacobi diagram $J$ to be the number of trivalent vertices of $J$.
Let $\A_n$ (resp.\ $\A_n^c$) denote the $\Z$-module generated by Jacobi diagrams (resp.\ connected Jacobi diagrams) of $\ideg =n$ subject to the AS, IHX, and self-loop relations:
\[
\Ygraph{0.5}{}{}{}
+
\begin{tikzpicture}[scale=0.5, baseline={(0,-0.1)}, densely dashed]
  \coordinate (origin) at (0,0);
  \draw (origin) .. controls +(1,0.5) and +(1,-0.5) .. (-1,1) node[at end, anchor=south east] {};
  \draw (origin) .. controls +(-1,0.5) and +(-1,-0.5) .. (1,1) node[at end, anchor=south west] {};
  \draw (origin) -- (0,-1) node[at end, anchor=north] {};
\end{tikzpicture}%
=0,
\quad
\Igraph{}{}{}{} - \Hgraph{}{}{}{} +
\begin{tikzpicture}[scale=0.5, baseline={(0,-0.1)}, densely dashed]
  \draw (-1,1) -- (1,-1);
  \draw (-1,-1) -- (1,1);
  \draw (-0.5,-0.5) -- (0.5,-0.5);
\end{tikzpicture}%
=0,
\quad
\begin{tikzpicture}[scale=0.5, baseline={(0,-0.1)}, densely dashed]
  \draw (0,0) circle [radius=1];
  \draw (1,0) -- (2,0) node {};
\end{tikzpicture}%
=0,
\]
where the rest of the diagrams are the same in each relation.
Note that the IHX relation implies the self-loop relation among connected Jacobi diagrams of $\ideg\geq 2$.

We obtain a homology cylinder from a Jacobi diagram by clasper surgery.
In fact, there is a homomorphism $\ss \colon \A_n^c \to Y_n\I\C/Y_{n+1}$, which is called the \emph{surgery map}.
We explain the details in Section~\ref{sec:refined}, and introduce refined versions of $\A_n$ and $\ss$.

\subsection{Maps $\bar{z}_{n+1}$ and $\delta$}
\label{subsec:NSS_invariant}
Cheptea, Habiro and Massuyeau~\cite{CHM08} introduced the LMO functor, which is a functorial extension of the LMO invariant of closed 3-manifolds.
The LMO functor induces a homomorphism $Y_n\I\C/Y_{n+1} \to \A_n^c \otimes \Q$, which is known to be an isomorphism over $\Q$ and the surgery map $\ss$ induces its inverse up to sign (\cite[Theorem~7.11]{CHM08}).
Using the LMO functor, the authors~\cite[Section~4]{NSS21} introduced a homomorphism $\bar{z}_{n+1}$ satisfying the commutative diagram
\[\xymatrix{
\A_n^c \ar[r]^-{\ss} \ar[d]_-{\delta} & Y_n\I\C/Y_{n+1} \ar[d]^-{\bar{z}_{n+1}} \\
\A_{n+1}^c\otimes\Z/2\Z \ar[r]^-{\id\otimes\frac{1}{2}} & \A_{n+1}^c\otimes\Q/\Z .
}\]
Here the homomorphism $\delta \colon \A_n^c \to \A_{n+1}^c\otimes\Z/2\Z$ is defined to be the sum of two maps $\delta'$ and $\delta''$ (\cite[Section~3.1]{NSS21}).
For the argument in Section~\ref{sec:refined}, we here redefine $\delta'$ and $\delta''$ as maps between modules of Jacobi diagrams without the AS, IHX, and self-loop relations.
We write $\J_{n}^c$ for the set of connected Jacobi diagrams $J$ of $\ideg(J)=n$.
Let $U(J)$ denote the set of univalent vertices of $J$ and let $\ell(v)$ be the label of $v \in U(J)$.
We temporarily fix a total order $\prec$ on each set $\{v \in U(J) \mid \ell(v)=i^\epsilon\}$, where $i^\epsilon \in \{1^\pm,\dots,g^\pm\}$.

\begin{definition}\label{def:delta}
Let $J \in \J_{n}^c$.
For $v \in U(J)$, let $\delta_v(J)$ denote the sum of two Jacobi diagrams
\[
\begin{tikzpicture}[scale=0.3, baseline={(0,-0.1)}, densely dashed]
 \draw (120:2) .. controls +(-1,0) and +(0,2) .. (180:4);
 \draw (-120:2) .. controls +(-1,0) and +(0,-2) .. (180:4);
 \fill[fill=gray!50] (0,0) circle [radius=2];
 \draw (-3.8,1) -- (-5,1) node[anchor=east] {$\ell(v)$};
 \draw (-3.8,-1) -- (-5,-1) node[anchor=east] {$\ell(v)$};
\end{tikzpicture}\ 
+
\begin{tikzpicture}[scale=0.3, baseline={(0,-0.1)}, densely dashed]
 \draw (120:2) .. controls +(-1,0) and +(0,2) .. (180:4);
 \draw (-120:2) .. controls +(-1,0) and +(0,-2) .. (180:4);
 \fill[fill=gray!50] (0,0) circle [radius=2];
 \draw (-4,0) -- (-5,0);
 \draw (-5,0) -- (-6,1) node[anchor=east] {$\ell(v)^\ast$};
 \draw (-5,0) -- (-6,-1) node[anchor=east] {$\ell(v)$};
\end{tikzpicture}
\quad
\Bigl(J=
\begin{tikzpicture}[scale=0.3, baseline={(0,-0.1)}, densely dashed]
 \draw (120:2) .. controls +(-1,0) and +(0,2) .. (180:4);
 \draw (-120:2) .. controls +(-1,0) and +(0,-2) .. (180:4);
 \fill[fill=gray!50] (0,0) circle [radius=2];
 \draw (-5,0) -- (-4,0) node[at start, anchor=east] {$\ell(v)$};
\end{tikzpicture}\ 
\Bigr),
\]
where $(i^\pm)^\ast = i^\mp$, and the shaded regions represent the same diagram.
Then $\delta' \colon \Z\J_{n}^c \to \Z\J_{n+1}^c$ is defined by $\delta'(J) = \sum_{v \in U(J)} \delta_v(J)$.
Next, for $v \prec w \in U(J)$, let $\delta_{vw}(J)$ denote the Jacobi diagram
\[
\begin{tikzpicture}[scale=0.3, baseline={(0,-0.1)}, densely dashed]
 \draw (120:2) .. controls +(-1,0) and +(0,2) .. (180:4);
 \draw (-120:2) .. controls +(-1,0) and +(0,-2) .. (180:4);
 \fill[fill=gray!50] (0,0) circle [radius=2];
 \draw (-5,0) -- (-4,0) node[at start, anchor=east] {$\ell(v)$};
\end{tikzpicture}
\quad
\Bigl(J=
\begin{tikzpicture}[scale=0.3, baseline={(0,-0.1)}, densely dashed]
 \draw (120:2) -- ($(120:2)+(-2,0)$) node[anchor=east] {$\ell(v)$};
 \draw (-120:2) -- ($(-120:2)+(-2,0)$) node[anchor=east] {$\ell(w)$};
 \fill[fill=gray!50] (0,0) circle [radius=2];
\end{tikzpicture}\ 
\Bigr),
\]
and $\delta'' \colon \Z\J_{n}^c \to \Z\J_{n+1}^c$ is defined by $\delta''(J) = \sum_{v \prec w} \delta_{vw}(J)$.
Note that $v \prec w \in U(J)$ implies $\ell(v)=\ell(w)$.
\end{definition}

The above $\delta'$ and $\delta''$ induce the homomorphisms $\delta', \delta'' \colon \A_n^c \to \A_{n+1}^c\otimes\Z/2\Z$ introduced in \cite[Section~4]{NSS21}.
Note that the induced map $\delta''$ is independent of the choice of the total order.

Finally note that the subscripts of the map $\ss_{n,l}$ and $\bar{z}_{n+1,k}$ are both based on information about Jacobi diagrams, that is, the subscript of the $\ss_{n,l}$ (resp.\ $\bar{z}_{n+1,k}$) is the same as its domain (resp.\ codomain).

\section{Refined surgery map and refined relations}
\label{sec:refined}
In this section, we give refinements of the surgery map and the AS and STU relations among claspers to introduce some relations in $Y_n\I\C/Y_{n+2}$ for $n\ge 2$.

\subsection{Refined surgery map}
Recall that the surgery map $\ss\colon \A_n^c\to Y_n\I\C/Y_{n+1}$ is a homomorphism defined in terms of clasper surgeries corresponding to Jacobi diagrams labeled by $\{1^{\pm},\ldots, g^{\pm}\}$ in the trivial homology cylinder.
See \cite[Theorem~8.8]{CHM08}, \cite[Section~8.5]{Hab00C}, and \cite[Theorem~4.13]{GGP01} for details. 
Here, we lift the target of the surgery map $\ss\colon \A_n^c\to Y_n\I\C/Y_{n+1}$ to $Y_n\I\C/Y_{n+2}$ when $n\ge2$ to obtain relations of $Y_n\I\C/Y_{n+2}$.

The ordinary surgery map $\ss$ is defined as follows.
For a connected Jacobi diagram $J$ labeled by $\{1^{\pm},2^{\pm},\ldots,g^{\pm}\}$,
we assign disks and annuli called nodes and leaves to trivalent and univalent vertices in $J$, respectively.
By gluing them to bands corresponding to edges, we obtain a compact surface. 
To embed the surface into $\Sigma_{g,1}\times [-1,1]$,
fix an orientation of the surface,
and identify $\Sigma_{g,1}\times [-1,1]$ with the bottom-top tangle $\Id_g$ depicted in \cite[Figure~2.6]{CHM08}.
First, we embed each leaf along a meridian of a component of the tangle which represents the label of the corresponding univalent vertex.
We assume that the annulus is vertical to the tangle,
and that the orientation coincides with that of a fiber of the normal bundle of the tangle.
Second, we embed the nodes in an arbitrary way,
and also embed the bands in the complement of $\Id_g$ in $[-1,1]^3$ so that the orientations of the constituents are compatible.
In this way, we obtain a graph clasper in $\Sigma_{g,1}\times [-1,1]$,
and we denote its $Y_{n+1}$-equivalence class by $\ss(J)$.
Here, we call a connected clasper without boxes a \emph{graph clasper}.

To lift the target of the surgery map, we consider Jacobi diagrams with labels that have information on the relative positions of leaves with respect to the orientation of the tangle and of half-twists on edges.

\begin{definition}\label{def:jacobi}
Let $J$ be a connected Jacobi diagram whose univalent vertices are labeled by the set
\[
\mathcal{L}=\{1_j^{\pm},\ldots, g_j^{\pm}, \bar{1}_j^{\pm},\ldots, \bar{g}_j^{\pm}\mid j\in\Z_{\ge1}\}
\]
satisfying the condition that, for each $i^{\epsilon}\in \{1^\pm,\ldots,g^\pm\}$,
$i_j^\epsilon$ or $\bar{i}_j^{\epsilon}$ appears as a label in $J$ exactly once for $j=1,2,\ldots, n(i^\epsilon)$ for some $n(i^\epsilon)\ge0$.
Let us denote by $\widetilde{\J}_n^c$ the set of such connected Jacobi diagrams $J$ with $\ideg(J)=n$.
As in the ordinary surgery map,
in the label $i_j^\epsilon$ or $\bar{i}_j^\epsilon$ for $1\le i\le g$ and $\epsilon=\pm$,
the symbol $i^\epsilon$ corresponds to a component of the bottom-top tangle $\Id_g$ in the cube.
The subscript $j$ of $i_j^\epsilon$ or $\bar{i}_j^\epsilon$ indicates the relative position of a leaf among the leaves with respect to the orientation of the component corresponding to the label $i^\epsilon$,
and the difference between the symbols $i_j^{\epsilon}$ and $\bar{i}_j^{\epsilon}$ is related to a positive half-twist on the edge incident to a leaf when realizing as a graph clasper.
\end{definition}

\begin{remark}
In \cite{Hab00C}, a module $\A_n(\Sigma_{g,1})$ of Jacobi diagrams is defined, which is isomorphic to $\A_n^c$.
Habiro used a total order on univalent vertices to denote the relative positions of leaves of a graph clasper.
Instead, we adopt subscripts of labels to denote the relative positions of leaves.
\end{remark}

We define the \emph{refined surgery map}
\[
\tilde{\ss}\colon\Z\widetilde{\J}_n^c\to Y_n\I\C/Y_{n+2}
\]
in a similar way as the ordinary surgery map by embedding the compact oriented surface $S$ corresponding to a connected Jacobi diagram.
The first difference is that we choose some standard homotopy classes of edges when we embed $S$ into the complement of the bottom-top tangle $\Id_g$ in $[-1,1]^3$ as follows.
First, we embed the $n(i^+)$ (resp.\ $n(i^-)$) copies of annuli corresponding to the univalent vertices labeled by $i_j^+$ and $\bar{i}_j^+$ (resp.\ $i_j^-$  and $\bar{i}_j^-$) in a small neighborhood of the terminal point of the component $i^+$ (resp.\ the initial point of the component $i^-$) of the tangle for each $i=1,2,\ldots,g$.
We assume that the positions of leaves are arranged in order of the subscripts of the labels with respect to the orientation of the tangle.
Second, we set disks as nodes in the subspace $\Int([-1,1]\times [-1,-1+\delta]\times [-1,1])$ of the cube, where $\delta>0$ is sufficiently small.
Corresponding to each edge in the Jacobi diagram incident to two trivalent vertices,
we connect two nodes by a band in $\Int([-1,1]\times [-1,-1+\delta]\times [-1,1])$ so that the orientations are compatible.
Lastly, to embed the rest of the edges,
we draw a line which is parallel to the second coordinate axis from each leaf to the plane $[-1,1]\times\{-1+\delta\}\times [-1,1]$,
and we connect the end points of lines in the plane to the corresponding nodes by arcs in $(-1,1)\times (-1,-1+\delta]\times (-1,1)$.
By fattening the lines and arcs to bands so that the orientations are compatible with the nodes and leaves,
we obtain embedded edges.

The second difference is that,
if the label of a univalent vertex is of the form $\bar{i}_j^\epsilon$,
we apply a positive half-twist when we connect the band to the corresponding leaf.

For elements $a_1,a_2,\ldots, a_n$ in $\mathcal{L}$ satisfying the condition in Definition~\ref{def:jacobi}
(resp.\ in $\{1^{\pm},2^{\pm},\ldots,g^{\pm}\}$),
let us denote
\[
T(a_1,a_2,\ldots, a_n)=
\begin{tikzpicture}[baseline=1.3ex, scale=0.25, dash pattern={on 2pt off 1pt}]
\node [left] at (0,0){$a_1$};
\node [above] at (2,2){$a_2$};
\node [above] at (4,2){$a_3$};
\node [below] at (6,2){$\cdots$};
\node [below] at (8.3,2){$\cdot$};
\node [below] at (9,2){$\cdot$};
\node [above] at (10,2){$a_{n-1}$};
\node [right] at (12,0){$a_n$,};
\draw (0,0) -- (12,0);
\draw (2,0) -- (2,2);
\draw (4,0) -- (4,2);
\draw (10,0) -- (10,2);
\end{tikzpicture}
\]
which is an element in $\Z\widetilde{\J}_{n-2}^c$ (resp.\ in $\A_{n-2}^c$).

\begin{example}\label{ex:refined_surgery}
The image of $T(1_1^-, 1_2^+, 1_1^+, \bar{2}_1^-)\in\Z\widetilde{\J}_2^c$ under the refined surgery map $\tilde{\ss}\colon\Z\widetilde{\J}_2^c\to Y_2\I\C_{2,1}/Y_4$ is
\[
\lss(T(1_1^-, 1_2^+, 1_1^+, \bar{2}_1^-))=\raisebox{-1.5cm}{\includegraphics[height=3cm]{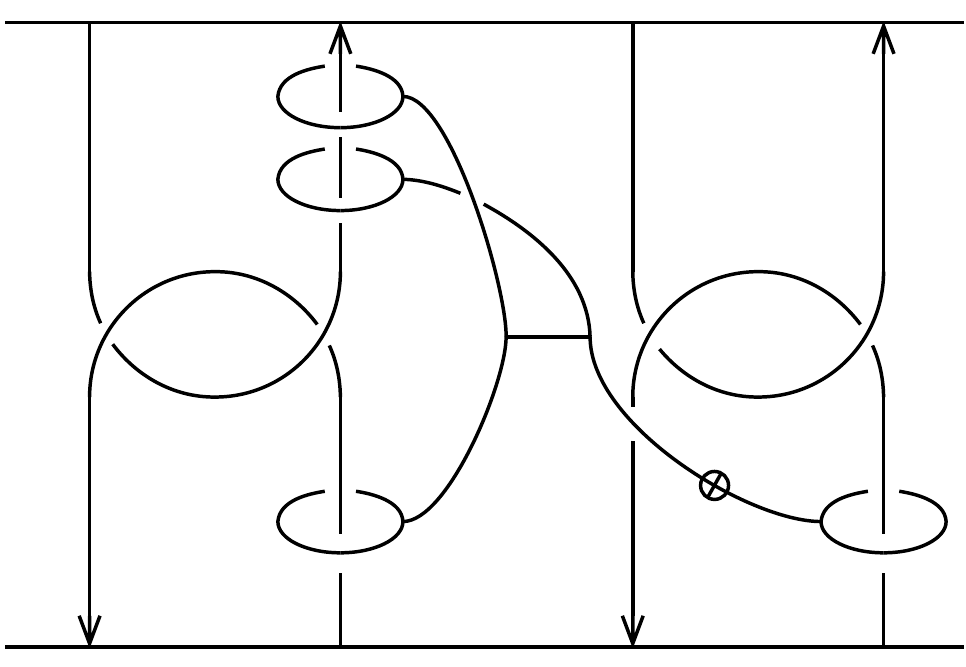}}\in Y_2\I\C_{2,1}/Y_4,
\]
where $\oplus$ in the figure denotes a positive half-twist.
\end{example}

\begin{lemma}\label{lem:refinedsurgery}
For $n\ge2$,
the refined surgery map $\tilde{\ss}\colon\Z\widetilde{\J}_n^c\to Y_n\I\C/Y_{n+2}$ is a well-defined homomorphism.
\end{lemma}
\begin{proof}
Before discussing $\tilde{\ss}$,
we recall that there are three ambiguities when we embed the surface corresponding to a connected Jacobi diagram labeled by $\{1^{\pm},\ldots, g^{\pm}\}$ in the definition of $\ss$;
framings of edges, isotopy classes of edges, and positions of leaves,
where the third ambiguity occurs only when multiple vertices have the same label.
Recall that the ordinary surgery map $\ss\colon\A_n^c\to Y_n\I\C/Y_{n+1}$ is well-defined because
the first ambiguity is eliminated by edge twists (\cite[Lemma~E.7]{Oht02}),
the second one is eliminated by edge slidings (\cite[Lemma~E.5]{Oht02}, \cite[Lemma~A.1]{MaMe13}), 
and the third one is also eliminated by leaf crossings (\cite[Lemma~E.6]{Oht02}).
If we try to lift the target of the surgery map from the $Y_{n+1}$-equivalence classes to the $Y_{n+2}$-equivalence classes,
full-twists along edges are still eliminated using \cite[Lemmas~A.1 and A.5]{MaMe13} with respect to a connected sum of an edge with a trivial knot with $(-1)$-framing.
With respect to the second ambiguity,
we cannot slide an edge of a graph clasper as in \cite[Lemma~A.1]{MaMe13},
although we can still apply a crossing change of edges in a graph clasper up to $Y_{n+2}$-equivalence in the same way as \cite[Proposition~4.6]{Hab00C}.
The proof is almost the same except that the subtree $T$ in the proof of \cite[Proposition~4.6]{Hab00C} is of degree $n+2$ (in our notation) after the zip construction.
Thus, we need to fix the homotopy classes of edges as we did in the refined surgery map.
Because of the third ambiguity,
we assigned the subscript $j$ in the labels of $J$,
and the positions of the leaves attaching to the same component of the tangle are ordered.

By edge slidings (\cite[Lemma~2.6(2)(a)]{Mei06}) and leaf crossings based on \cite[Lemma~2.6(2)(b)]{Mei06} of two graph claspers,
we can also show that $\tilde{\ss}$ is a homomorphism in the same way as $\ss$.
\end{proof}

\begin{remark}
Full twists along edges of a graph clasper of degree $n$ do not change its $Y_{n+2}$-equivalence class.
Thus, it does not matter whether we apply a positive or negative half-twist to an edge in the refined surgery map $\lss$
when the label of a univalent vertex is of the form $\bar{i}^{\epsilon}$.
\end{remark}

Define a map
\[
\varpi\colon \mathcal{L}\to \{1^{\pm},\ldots, g^{\pm}\},
\]
by $i_j^{\pm}\mapsto i^{\pm}$ and $\bar{i}_j^{\pm}\mapsto i^{\pm}$.
We also define a homomorphism
\[
\mathcal{P}\colon \Z\widetilde{\J}_n^c\to\A_n^c
\]
by $\mathcal{P}(J)=(-1)^k\bar{J}$,
where $\bar{J}$ is the Jacobi diagram obtained by changing the labels of the univalent vertices in $J\in \Z\widetilde{\J}_n^c$ by $\varpi$,
and $k$ is the number of univalent vertices in $J$ colored by $\{\bar{1}_j^{\pm},\ldots, \bar{g}_j^{\pm}\mid j\in\Z_{\ge1}\}\subset\mathcal{L}$.
The AS relation in $\A_n^c$ implies that if we apply a half-twist to an edge in a graph clasper representing $\ss(\bar{J})$, then the sign of its $Y_{n+1}$-equivalence class changes.
Thus, we have a commutative diagram
\[
\xymatrix{
\Z\widetilde{\J}_n^c \ar[r]^-{\lss} \ar@{->>}[d]_-{\mathcal{P}} & Y_n\I\C/Y_{n+2} \ar@{->>}[d] \\
\A_n^c \ar[r]^-{\ss} & Y_n\I\C/Y_{n+1}
}\]
for $n\ge2$,
where the right vertical map is the natural projection.

\subsection{Refined AS and refined STU relations}
Here, we introduce the refined AS and refined STU relations,
and introduce some relations in $Y_n\I\C/Y_{n+2}$.

\begin{lemma}\label{lem:AS}
Let $n\ge1$, and let $G$ be a graph clasper with $n$ nodes in $\Sigma_{g,1}\times[-1,1]$.
We denote by $G'$ the graph clasper obtained from $G$ by inserting a positive half-twist in one edge $e$.

If $e$ is incident to a leaf and a node in $G$,
denote by $H$ the graph clasper obtained by doubling a neighborhood of $e$ as below.
Then, we have
\[
(\Sigma_{g,1}\times[-1,1])_G\circ (\Sigma_{g,1}\times[-1,1])_{G'}\circ (\Sigma_{g,1}\times[-1,1])_H
\sim_{Y_{n+2}} \Sigma_{g,1}\times[-1,1].
\]
\begin{center}
\includegraphics[height=2.5cm]{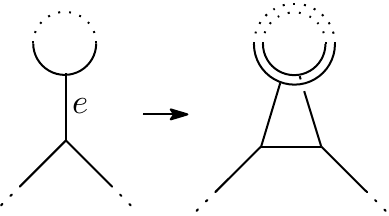}
\end{center}
If $e$ is incident to two nodes in $G$,
we have 
\[
(\Sigma_{g,1}\times[-1,1])_G\circ (\Sigma_{g,1}\times[-1,1])_{G'}
\sim_{Y_{n+2}} \Sigma_{g,1}\times[-1,1].
\]
\end{lemma}

\begin{proof}
\newcommand{\hidari}[1]{
\begin{scope}[xshift=#1]
 \draw [dotted] (0,11) -- (1,11); 
 \draw (1,11)--(2,11);
 \draw (2,9) rectangle (3,13);
 \draw (3,12)--(5,12);
 \draw (3,10)--(5,10);
 \draw (5.5,10) arc (0:160:0.5);
 \draw (5.5,10) arc (0:-160:0.5);
 \draw (5,9.5)--(5,8);
 \draw (5,5)--(5,3.5);
 \draw (4.5,3) arc (180:20:0.5);
 \draw (4.5,3) arc (-180:-20:0.5);
 \draw (3,3)--(4.2,3);
 \draw (3,1)--(5,1);
 \draw [dotted] (0,2) -- (1,2);
 \draw (1,2)--(2,2);
 \draw (2,0) rectangle (3,4);
\end{scope}
}
\newcommand{\migi}[1]{
\begin{scope}[xshift=#1]
 \draw (0,12)--(2,12);
 \draw (0,10)--(2,10);
 \draw (2,9) rectangle (3,13);
 \draw (3,11)--(4,11);
 \draw [dotted] (4,11) -- (5,11);
 \draw (0,3)--(2,3);
 \draw (0,1)--(2,1);
 \draw (2,0) rectangle (3,4);
 \draw (3,2)--(4,2);
 \draw [dotted] (4,2) -- (5,2);
 \draw (1,3)--(1,10);
 \draw (1,6.5) circle [radius=0.3];
 \draw (0.7,6.5)--(1.3,6.5);
\end{scope}
}
Assume that $e$ is incident to a leaf and a node.
In \cite[Lemma~A.9]{MaMe13}, the case $n=1$ is treated.
The proof is based on their clasper calculus.
If we change a neighborhood of $e$ in $G$ as below,
it is equivalent to the empty one by Move~4 in \cite{Hab00C}.
\begin{center}
\includegraphics[height=2.7cm]{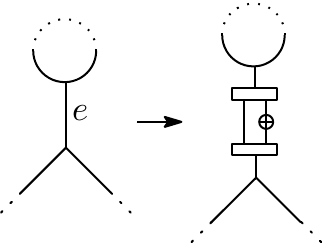}
\end{center}

It is also equivalent to the left-hand side below as in the proof of \cite[Lemma~A.9]{MaMe13},
where $\ominus$ denotes a negative half-twist.
\begin{center}
\includegraphics[height=3.5cm]{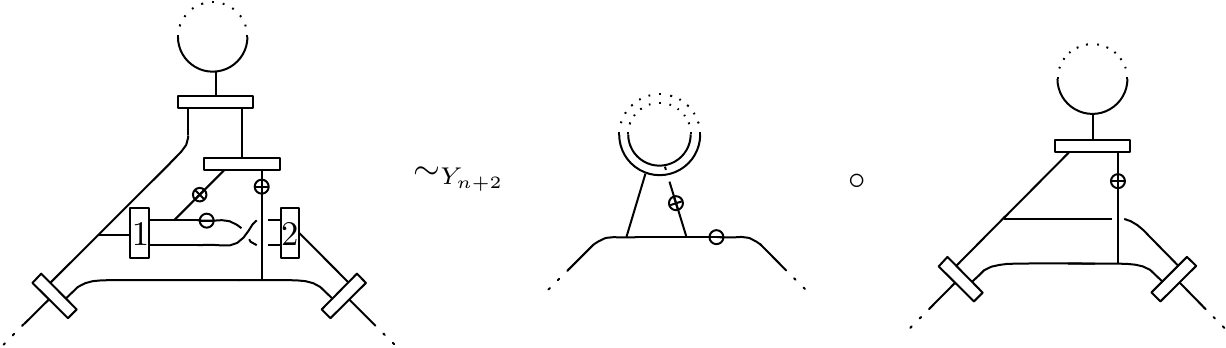}
\end{center}
By the zip construction with the upper input of Box~1 and the lower input of Box~2 as a marking and by applying Move~5 in \cite{Hab00C} several times,
we obtain a graph clasper of degree $n+1$, which is the first one on the right-hand side.
By edge slidings (\cite[Lemma~2.6(2)(a)]{Mei06}) and leaf crossings based on \cite[Lemma~2.6(2)(b)]{Mei06},
we can separate it from the rest up to $Y_{n+2}$-equivalence,
and the corresponding homology cylinder decomposes into a product of two homology cylinders as in the right-hand side.
By \cite[Lemma~E.7 or (E.10)]{Oht02},
the first one is represented by $(\Sigma_{g,1}\times[-1,1])_H$ up to $Y_{n+2}$-equivalence.
In the second homology cylinder,
apply the zip construction with the left input of the lower-left box and the right input of the lower-right box as a marking.
Then, it decomposes into the product $(\Sigma_{g,1}\times[-1,1])_G\circ (\Sigma_{g,1}\times[-1,1])_{G'}$.
Here, to erase extra boxes arising from the zip construction,
we applied Move~12 to the leaf attached to an input of each extra box and the zip construction,
and erased graph claspers of degree $n+2$.
We also used Move~11 on extra boxes to increase the degrees of the graph claspers.
It is also possible to erase extra boxes by applying Move~11, crossing changes (\cite[Lemma~2.6(2)(b)]{Mei06}) of edges and leaves, and Move~3.
If the graph clasper $G$ has loops, there arise pairs of boxes whose outputs are connected by one edge.
We also used the move \cite[Figure~38]{Hab00C} and the zip construction to erase these boxes.

Next, we treat the case where $e$ is incident to two nodes.
In the same way,
consider the clasper obtained by changing a neighborhood of $e$ as on the left-hand side below.
\[
\begin{tikzpicture}[scale=0.25, baseline={(0,1.25)}]
 \draw (1,5.5) circle [radius=0.3];
 \draw (0.7,5.5)--(1.3,5.5);
 \draw (1,7)--(1,4);
 \draw (-1,7)--(-1,4);
 \draw (-2,7) rectangle (2,8);
 \draw (0,8)--(0,10);
 \draw (0,10)--(-2,10);
 \draw [dotted] (-2,10)--(-3,10);
 \draw (0,10)--(2,10);
 \draw [dotted] (2,10)--(3,10);
 \draw (-2,3) rectangle (2,4);
 \draw (0,3)--(0,1);
 \draw (0,1)--(-2,1);
 \draw [dotted] (-2,1)--(-3,1);
 \draw (0,1)--(2,1);
 \draw [dotted] (2,1)--(3,1);
\end{tikzpicture}
\quad\sim\quad
\begin{tikzpicture}[scale=0.25, baseline={(0,1.5)}]
 \hidari{-5cm}
 \migi{3cm}
 \draw (-1,7) rectangle (3,8);
 \draw (-1,5) rectangle (3,6);
 \draw (0,12)--(3,12);
 \draw (0.8,10)--(1.8,10);
 \draw (2.2,10)--(3,10);
 \draw (0,3)--(3,3);
 \draw (0,1)--(3,1);
 \draw (2,12)--(2,8);
 \draw (2,5)--(2,3.2);
 \draw (2,2.8)--(2,1);
 \draw (1,7)--(1,6);
\end{tikzpicture}
\]
Using Move~11 twice, we have the clasper on the right-hand side.
By the move \cite[Figure~38]{Hab00C} and Move~6,
it is equivalent to the clasper on the left-hand side below.
\[
\begin{tikzpicture}[scale=0.25, baseline={(0,1.5)}]
 \hidari{-5cm}
 \migi{5cm}
 \draw (2.5,10) arc (0:160:0.5);
 \draw (2.5,10) arc (0:-160:0.5);
 \draw (0.8,10)--(2,10);
 \draw (0,3)--(1.2,3);
 \draw (1.5,3) arc (180:20:0.5);
 \draw (2,9.5) .. controls +(0,-1) and +(-0.5,1) .. (2.8,6.7);
 \draw (3.2,6) .. controls +(0.1,-0.2) and +(0,0.2) .. (3.5,5);
 \draw (1.5,3) arc (-180:-20:0.5);
 \draw (3.5,8) .. controls +(0,-1) and +(0,1) .. (2,3.5);
 \draw (3,8) rectangle (5,9);
 \node at (4,8.5) {\Small 1};
 \draw (4,9)--(4,12);
 \draw (3,4) rectangle (5,5);
 \node at (4,4.5) {\Small 2};
 \draw (4,4)--(4,3.2);
 \draw (4,2.8)--(4,1);
 \draw (0,12)--(5,12);
 \draw (2.8,10)--(3.8,10);
 \draw (4.2,10)--(5,10);
 \draw (2,3)--(5,3);
 \draw (0,1)--(5,1);
 \draw (0,5)--(0,8);
 \draw (4.5,5)--(4.5,8);
\end{tikzpicture}
\quad\sim_{Y_{n+2}}\quad
\begin{tikzpicture}[scale=0.25, baseline={(0,1.5)}]
 \hidari{-5cm}
 \migi{5cm}
 \draw (0,12)--(5,12);
 \draw (0.8,10)--(3.8,10);
 \draw (4.2,10)--(5,10);
 \draw (0,3)--(5,3);
 \draw (0,1)--(5,1);
 \draw (4,2.8)--(4,1);
 \draw (0,5)--(0,8);
 \draw (4,12)--(4,3.2);
\end{tikzpicture}
\]
By the zip construction with the four edges incident to the twisted edge as a marking,
we obtain a graph clasper of degree $n$ as a subset of the resulting clasper.
Thus, after applying Move~11 to Boxes~1 and 2 so that each of the two leaves attached to the left inputs is incident to a node, 
we can pass horizontal edges in the figure across the two leaves up to $Y_{n+2}$-equivalence by \cite[Lemma~2.6(2)(b)]{Mei06}.
Applying Move~3 to erase Boxes $1$ and $2$ and the same zip construction backward, we obtain the clasper on the right-hand side.
The graph clasper with no node and two leaves in the figure corresponds to a crossing change of two edges,
and we can also erase it up to $Y_{n+2}$-equivalence as we explained in the proof of Lemma~\ref{lem:refinedsurgery}.
After applying the same zip construction, 
we can separate the graph clasper of degree $n$ from the rest by edge slidings \cite[Lemma~2.6(2)(a)]{Mei06} and leaf crossings \cite[Lemma~2.6(2)(b)]{Mei06},
and the homology cylinder decomposes into the product $(\Sigma_{g,1}\times[-1,1])_G\circ (\Sigma_{g,1}\times[-1,1])_{G'}$.
\end{proof}

Let $J\in\widetilde{\J}_n^c$ be a connected Jacobi diagram,
and let $v$ be one of its univalent vertices.
Similar to the first term of $\delta_v(J)$ in Definition~\ref{def:delta},
we consider the Jacobi diagram $J_v$ obtained by doubling a neighborhood of the edge incident to $v$,
and change the labels of all the univalent vertices by the map $\varpi$.
Let us denote by 
\[
\tilde{\delta}_v(J)=(-1)^kJ_v\in\A_{n+1}^c, 
\]
where $k$ is the number of univalent vertices in $J$, except $v$, labeled by $\{\bar{1}_j^{\pm},\ldots, \bar{g}_j^{\pm}\mid j\in\Z_{\ge1}\}$.
By the former half of Lemma~\ref{lem:AS}, we obtain the following.

\begin{corollary}[refined AS relation]\label{cor:AS}
Let $n\ge2$.
For a connected Jacobi diagram $J\in \widetilde{\J}_n^c$,
we denote by $J'$ the Jacobi diagram obtained by changing the label of one univalent vertex $v$ in $J$ as $\ell(v)\mapsto \overline{\ell(v)}$.
Then, we have
\[
\lss(J)+\lss(J')+\ss(\tilde{\delta}_v(J))=0\in Y_n\I\C/Y_{n+2}.
\]
\end{corollary}

Note that Corollary~\ref{cor:AS} implies the move in \cite[Lemma~E.9]{Oht02}, which corresponds to the AS relation in $\A_n^c$.
For $n\ge2$, let us define a homomorphism
\[
\rev\colon \Z\widetilde{\J}_n^c \to \Z\widetilde{\J}_n^c
\]
by the map which changes the cyclic order of every trivalent vertex to the other one,
and changing the label $a\in\mathcal{L}$ of each univalent vertex into $\bar{a}$.
Here, we regard $\bar{\bar{a}}$ as $a$.
Note that we have
\begin{align}
\lss(\rev(J))=\lss(J) \label{eq:rev}
\end{align}
for $J\in\widetilde{\J}_n^c$
because if we reverse neighborhoods of all the nodes of the graph clasper corresponding to $\lss(J)$,
we obtain the graph clasper corresponding to $\lss(\rev(J))$.

Using Corollary~\ref{cor:AS}, we obtain relations in $Y_n\I\C/Y_{n+2}$.
\begin{example}
Let $a_1,a_2,\ldots, a_n\in \{1_1^{\pm},\ldots, g_1^{\pm}\}\subset \mathcal{L}$ be mutually distinct elements.
We have
\begin{align*}
\lss(T(a_1,a_2,\ldots, a_n))
&=\lss(\rev(T(a_1,a_2,\ldots, a_n)))\\
&=\lss(T(\bar{a}_n\ldots,\bar{a}_2, \bar{a}_1))\\
&=(-1)^n\sum_{i=1}^n \ss(T(\varpi(a_n),\ldots, \varpi(a_{i+1}), \varpi(a_i), \varpi(a_i), \varpi(a_{i-1}),\ldots, \varpi(a_1)))\\
&\quad\qquad+(-1)^n\lss(T(a_n,\ldots, a_2, a_1)).
\end{align*}
\end{example}

Next, we introduce the refined STU relation.
The following lemma is well-known.
See the proof of Proposition~4.4 and Figure~29(a)--(f) in \cite{Hab00C}.
See also \cite[Lemma~2.2(1)]{MeYa12}.

\begin{lemma}\label{lem:STU}
Let $n\ge1$,
and let $G$ be a graph clasper with $n$ nodes in a $3$-manifold $M$ whose two leaves are locally described as in the left-hand side below,
where the thick lines imply the same bundle which may contain edges and leaves of claspers.
Let $G'$ be the graph clasper obtained from $G$ by swapping the positions of the two leaves as in the middle,
and let $H$ be the graph clasper obtained by merging two leaves into one leaf as in the right-hand side.
Then, we have
\[
M_{G'}\sim_{Y_{n+2}} M_{G\cup H}.
\]
\begin{center}
\includegraphics[height=2.5cm]{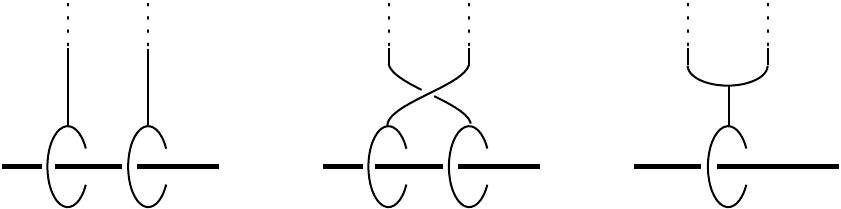}
\end{center}
\end{lemma}

Let $J\in\widetilde{\J}_n^c$ be a connected Jacobi diagram, and let $v$ and $w$ be two univalent vertices of $J$ with labels $i_j^\epsilon$ and $i_{j+1}^\epsilon$ for some $j=1,2,\ldots, n(i^\epsilon)-1$, respectively.
Similar to $\delta_{vw}(J)$ in Definition~\ref{def:delta},
we consider the Jacobi diagram $J_{vw}$ obtained by gluing a $Y$-shaped graph along the vertices $v$ and $w$ in $J$.
We set the label of the other univalent vertex in the $Y$-shaped graph by $i^\epsilon$ and change the labels of the other univalent vertices in $J_{vw}$ by the map $\varpi$.
Here, note that we choose the cyclic order of the trivalent vertex of the $Y$-shaped graph as in the figure of $\delta_{vw}(J)$.
Let us denote by
\[
\tilde{\delta}_{vw}(J)=(-1)^kJ_{vw}\in\A_{n+1}^c,
\]
where $k$ is the number of univalent vertices in $J$ labeled by $\{\bar{1}_j^{\pm},\ldots, \bar{g}_j^{\pm}\mid j\in\Z_{\ge1}\}$.
By Lemma~\ref{lem:STU}, we obtain the following.

\begin{corollary}[refined STU relation]
Let $n\ge2$,
and let $J\in \widetilde{\J}_n^c$ be a connected Jacobi diagram with two vertices $v$ and $w$ labeled by $i_j^\epsilon$ and $i_{j+1}^\epsilon$ for some $j=1,2,\ldots, n(i^\epsilon)-1$, respectively.
We denote by $J'$ the Jacobi diagram obtained by exchanging the labels of $v$ and $w$ in $J$.
Then, we have
\[
\lss(J')=\lss(J)+\ss(\tilde{\delta}_{vw}(J))\in Y_n\I\C/Y_{n+2}.
\]
\end{corollary}

\subsection{Applications of refined relations and symmetric and reversible Jacobi diagrams}
\label{subsec:App_of_refined}
Here, we give some relations among clasper surgeries along reversible and symmetric Jacobi diagrams defined below, and prove Theorem~\ref{thm:Y3C/Y5}.

For a connected Jacobi diagram $J$ labeled by $\{1^{\pm},\ldots, g^{\pm}\}$,
let us denote by $J'$ the Jacobi diagram obtained by reversing the cyclic order of every trivalent vertex to the other one.
We call $J$ \emph{reversible} if $J'$ is isomorphic to $J$ as a Jacobi diagram,
namely as a uni-trivalent graph endowed with cyclic orders on the trivalent vertices and labels on the univalent vertices.
Note that, when $n=\ideg J$ is odd, we obtain $2J=0\in\A_{n}^c$ by the AS relation.
Forgetting the cyclic orders,
we can identify $J'$ with $J$,
and the above isomorphism $J'\to J$ gives an isomorphism $r\colon J\to J$ of uni-trivalent graphs,
not of Jacobi diagrams.
We call it a \emph{reversing map} of $J$.
Note that the Jacobi diagram $J=\theta(a,b; c,c; b,a)$ satisfies $J=J' \in \A_8^c$, but $J$ is not reversible if $a \neq b$.

A connected Jacobi diagram $J$ labeled by $\{1^{\pm},\ldots, g^{\pm}\}$ is called \emph{symmetric} if it is symmetric along some line when depicted in a diagrammatic form in $\mathbb{R}^2$.
For example,
the Jacobi diagrams $T(a_1,a_2,\ldots,a_n,a_n,\ldots,a_2,a_1)$ and $T(a_1,a_2,\ldots,a_n,\ldots,a_2,a_1)$ are symmetric for $n\ge2$ with respect to the center lines when depicted as the one just before Example~\ref{ex:refined_surgery}.
Let $\A_{n}^{c,s}$ denote the submodule of $\A_{n}^{c}$ generated by symmetric Jacobi diagrams, which decomposes as $\bigoplus_{l\geq 0}\A_{n,l}^{c,s}$.

Symmetric Jacobi diagrams are reversible, and the line symmetry is an involutive reversing map.
The next lemma asserts that the converse is also true.
Note that a reversing map is not necessarily involutive.
The following Jacobi diagram has a reversing map $r$ satisfying $r^2\ne 1$ and $r^4=1$:
\[
\begin{tikzpicture}[scale=0.25, baseline={(0,-0.1)}, densely dashed]
 \draw (0,0) circle [radius=4];
 \draw (0:4)--(0:6) node[anchor=west] {$a$};
 \draw (45:4)--(45:2) node[anchor=north] {$b$};
 \draw (90:4)--(90:2) node[anchor=north] {$a$};
 \draw (135:4)--(135:6) node[anchor=east] {$b$};
 \draw (180:4)--(180:6) node[anchor=east] {$a$};
 \draw (225:4)--(225:2) node[anchor=south] {$b$};
 \draw (270:4)--(270:2) node[anchor=south] {$a$};
 \draw (315:4)--(315:6) node[anchor=west] {$b$};
\end{tikzpicture}\ .
\]

\begin{lemma}
A reversible Jacobi diagram $J$ with reversing map $r\colon J\to J$ satisfying $r^2=1$ is symmetric.
\end{lemma}

\begin{proof}
Let us denote the fixed point set of $r$ by $\Fix(r)\subset J$.
Since $r$ changes the cyclic order of each trivalent vertex,
we see that $\Fix(r)$ is a disjoint union of edges and the midpoints of edges.

We embed $\Fix(r)$ in $\{0\}\times \mathbb{R}\subset\mathbb{R}^2$ in an arbitrary way,
and divide the rest of the uni-trivalent vertices contained in $J\setminus\Fix(r)$ into two subsets $V_1\sqcup V_2$ satisfying $r(V_1)=V_2$.
Set each point $v\in V_1$ in $\mathbb{R}_{<0}\times \mathbb{R}$ and $r(v)\in V_2$ in its symmetric point about the second axis.
Connect the vertices by edges so that the actions of the line symmetry about the second axis and $r$ to the set of edges coincides.
Also, note that the line symmetry must change the cyclic order of every trivalent vertex.
Then, we obtain a diagrammatic form of $J$ with the reversing map realized as the line symmetry.
\end{proof}

Let us denote
\[
O(a_1,a_2,a_3,\ldots, a_n) =
\begin{tikzpicture}[baseline=0.5ex, scale=0.25, densely dashed]
\draw (0,0) circle [radius=2]; 
\draw (10:2) -- (10:4);
\draw (50:2) -- (50:4);
\draw (90:2) -- (90:4);
\draw (130:2) -- (130:4);
\node [right] at (10:3.8) {$a_3$};
\node [above right] at (50:3.6) {$a_2$};
\node [above] at (90:3.8) {$a_1$};
\node [above left] at (130:3.6) {$a_n$};
\node at (-10:3){$\cdot$};
\node at (-25:3){$\cdot$};
\node at (-40:3){$\cdot$};
\node at (150:3){$\cdot$};
\node at (165:3){$\cdot$};
\end{tikzpicture}\ ,
\]
where $a_i$'s are in $\mathcal{L}$ or $\{1^{\pm},\ldots,g^{\pm}\}$.

\begin{remark}
\label{rem:symmetric}
We slightly extend the definition of a symmetric Jacobi diagram from \cite{NSS21}.
In \cite{NSS21}, we considered symmetries of $1$-loop Jacobi diagrams only of the forms
\begin{gather}
O(a_1,a_2,\ldots,a_m,a_m,\ldots, a_2,a_1),\
O(a_1,a_2,\ldots,a_{m},a_{m+1},a_{m},\ldots,a_2), \label{eq:evensymmetric}\\
O(a_1,a_2,\ldots,a_m,\ldots,a_2,a_1) \label{eq:oddsymmetric}
\end{gather}
for $a_i\in\{1^{\pm},\ldots,g^{\pm}\}$.
In other words,
we considered symmetries of 1-loop Jacobi diagrams whose trivalent vertices are not attached to rooted trees except struts.
\end{remark}

Let us denote by $B_n^s$ the set of
Jacobi diagrams of the forms \eqref{eq:evensymmetric} (resp.\ \eqref{eq:oddsymmetric}) in Remark~\ref{rem:symmetric} when $n=2m$ (resp.\ $n=2m-1$),
which were called symmetric in \cite{NSS21}.
We also denote by $\ang{B_n^s}$ the submodule of $\A_{n,1}^c$ generated by $B_n^s$, which was denoted by $\A_{n,1}^{c,s}$ in \cite{NSS21}.
In this paper, we also consider
\begin{equation}\label{fig:unexpected_sym}
\begin{tikzpicture}[scale=0.3, baseline={(0,-0.2)}, densely dashed]
 \draw (0,0) circle [radius=2];
 \draw (150:2) -- (150:3);
 \draw (150:3) -- (160:4) node[anchor=east] {$a$};
 \draw (150:3) -- (140:4) node[anchor=east] {$b$};
 \draw (-150:2) -- (-150:3);
 \draw (-150:3) -- (-160:4) node[anchor=east] {$a$};
 \draw (-150:3) -- (-140:4) node[anchor=east] {$b$};
 \draw (30:2) -- (30:3) node[anchor=west] {$c$};
 \draw (-30:2) -- (-30:3) node[anchor=west] {$c$};
\end{tikzpicture}
\end{equation}
as a symmetric Jacobi diagram, 
which is equal to
\[
O(a,b,c,c,b,a) + O(b,a,c,c,a,b) -2O(a,b,c,c,a,b) \in \A_{6,1}^c 
\]
under the AS and IHX relations. 
Also, note that $O(a,b,c,a,b,c) \in \A_{6,1}^{c,s}$ since one can find symmetry by flipping the half of the circle and by using the AS relation.

\begin{remark}
\label{rem:OequalA}
Suppose that $n$ is odd.
Since $\ang{B_n^s}=\tor \A_{n,1}^c$ as shown in \cite[Proposition~5.2]{NSS21},
we see that $\ang{B_n^s}=\A_{n,1}^{c,s}$.
\end{remark}

Let $J$ (resp.\ $\tilde{J}$) be a connected Jacobi diagram of $\ideg=n$ labeled by $\{1^{\pm},\ldots, g^{\pm}\}$ (resp.\ $\{1_j^{\pm},\ldots, g_j^{\pm}\mid j\in\Z_{\ge1}\}\subset\mathcal{L}$ satisfying the condition in Definition~\ref{def:jacobi}).
We call $\tilde{J}$ a \emph{lift} of $J$ if $\tilde{J}$ turns into $J$ when we change the labels of $\tilde{J}$ by $\varpi$.
In this case, the $Y_{n+2}$-equivalence class $\lss(\tilde{J})\in Y_n\I\C/Y_{n+2}$ is a lift of $\ss(J)\in Y_n\I\C/Y_{n+1}$.
For a symmetric Jacobi diagram $J$ with line symmetry $r\colon J\to J$,
we also call $\tilde{J}$ a \emph{good lift} of $J$ with respect to $r$ if the difference between the subscripts of the labels $\ell(r(v))$ and $\ell(v)$ of $\tilde{J}$ in $\mathcal{L}$ is at most one for every univalent vertex $v\in U(J)$.
It is obtained by choosing consecutive numbers as the subscripts of the labels in $\mathcal{L}$ for each pair $(v, r(v))$ of univalent vertices such that $v\ne r(v)$.

Using the refined AS and refined STU relations,
we obtain some elements in the kernels of $\ss\colon \A_n^c\to Y_n\I\C/Y_{n+1}$ and $\lss\colon \Z\widetilde{\J}_n^c \to Y_n\I\C/Y_{n+2}$.

\begin{theorem}\label{thm:symrelation}
Let $n\ge2$,
and let $J$ be a symmetric Jacobi diagram of $\ideg=n$ labeled by $\{1^{\pm},\ldots, g^{\pm}\}$ with line symmetry $r\colon J\to J$.
Let us denote by $\tilde{J}\in\Z\widetilde{\J}_n^c$ a good lift of $J\in\A_n^c$ with respect to $r$, and denote
\[
U^-(J)=\{v\in U(J) \mid \text{the subscript of $\ell(v)$ of $\tilde{J}$ is lower than that of $\ell(r(v))$}\}.
\]
Then, we have the following.
\begin{enumerate}
\item \label{item:even}
When $n$ is even,
\[
\sum_{\substack{v\in U(J)\\r(v)=v}} \ss(\tilde{\delta}_v(\tilde{J}))
+\sum_{v\in U^-(J)}\ss(\tilde{\delta}_{v\, r(v)}(\tilde{J}))=0\in Y_{n+1}\I\C/Y_{n+2}.
\]
\item \label{item:odd}
When $n$ is odd,
\[
2\lss(\tilde{J})=
-\sum_{v\in U(J)}\ss(\tilde{\delta}_v(\tilde{J}))
-\sum_{v\in U^-(J)}\ss(\tilde{\delta}_{v\, r(v)}(\tilde{J}))\in Y_{n}\I\C/Y_{n+2}.
\]
\end{enumerate}
\end{theorem}

\begin{proof}
The line symmetry $r\colon J\to J$ implies that
the Jacobi diagram $\rev(\tilde{J})$ is isomorphic to $\tilde{J}$ if we map the labels of each Jacobi diagram by $\varpi$.
Since $\tilde{J}$ is a lift,
all the labels of $\tilde{J}$ and $\rev(\tilde{J})$ are in $\{1_j^{\pm},\ldots, g_j^{\pm}\mid j\in\Z_{\ge1}\}$ and $\{\bar{1}_j^{\pm},\ldots, \bar{g}_j^{\pm}\mid j\in\Z_{\ge1}\}$, respectively.
Change all the labels of univalent vertices in $\rev(\tilde{J})$ as $a\mapsto \bar{a}$,
and denote it by $\tilde{J}_1$.
By the refined AS relation among claspers and the AS relation, we have
\begin{equation}\label{eq:sym2}
\lss(\rev(\tilde{J})) = (-1)^n\sum_{v\in U(J)}\ss(\tilde{\delta}_v(\tilde{J}))+(-1)^n\lss(\tilde{J}_1)\in Y_n\I\C/Y_{n+2}.
\end{equation}

The difference between $\tilde{J}_1$ and $\tilde{J}$ is only the subscripts of the labels.
Thus, if we exchange the subscripts of the labels of all $v\in U^-(J)$ in $\tilde{J}_1$ with that of $r(v)$,
we obtain $\tilde{J}$.
Since $\tilde{J}$ is a good lift,
the subscript of $\ell(v)$ in $\tilde{J}$ is lower than that of $\ell(r(v))$ by one for $v\in U^-(J)$.
By the refined STU relation, we have
\begin{equation}\label{eq:sym3}
\lss(\tilde{J}_1)=\sum_{v\in U^-(J)}\ss(\tilde{\delta}_{v\,r(v)}(\tilde{J}))+\lss(\tilde{J})\in Y_n\I\C/Y_{n+2}.
\end{equation}

By the equalities \eqref{eq:rev}, \eqref{eq:sym2}, and \eqref{eq:sym3}, we obtain
\[
(1-(-1)^n)\lss(\tilde{J})
=(-1)^n\sum_{v\in U(J)}\ss(\tilde{\delta}_v(\tilde{J}))+(-1)^n\sum_{v\in U^-(J)}\ss(\tilde{\delta}_{v\,r(v)}(\tilde{J})).
\]
When $n$ is odd, the conclusion follows from this equality.
Consider the case where $n$ is even.
The line symmetry $r$ gives an isomorphism between the uni-trivalent graphs each of which represents $\tilde{\delta}_v(\tilde{J})$ and $\tilde{\delta}_{r(v)}(\tilde{J})$ for $v\in U(J)$, respectively.
Thus, we have $\tilde{\delta}_v(\tilde{J})=-\tilde{\delta}_{r(v)}(\tilde{J})$ by the AS relation.
It implies that
\[
\sum_{v\in U(J)}\ss(\tilde{\delta}_v(\tilde{J}))
=\sum_{\substack{v\in U(J)\\r(v)=v}} \ss(\tilde{\delta}_v(\tilde{J})),
\]
and the conclusion follows.
\end{proof}

\begin{remark}
In the notation of Theorem~\ref{thm:symrelation},
$\tilde{\delta}_{v\, r(v)}(\tilde{J})\in\A_{n+1}^c$ for $v\in U^{-}(J)$ is represented by a symmetric Jacobi diagram.
Thus, we have $2\tilde{\delta}_{v\, r(v)}(\tilde{J})=0$ when $n$ is even.
We also have 
\[
\tilde{\delta}_{v\, r(v)}(\tilde{J})
=-\tilde{\delta}_{r(v)\, v}(\tilde{J})\in\A_{n+1}^c
\]
by the definition of $\tilde{\delta}_{v\, r(v)}(\tilde{J})$ and the AS relation.
Thus, the relator obtained in Theorem~\ref{thm:symrelation}(\ref{item:even}) is essentially independent of the choice of a good lift.
\end{remark}

For elements $a_i, b_j, c_k$ in $\mathcal{L}$ or $\{1^{\pm},\ldots,g^{\pm}\}$,
let us denote
\begin{equation}
\theta(a_1,\dots,a_p; b_1,\dots,b_q; c_1,\dots,c_r) =
\thetapqr\ ,
\label{eq:theta}
\end{equation}
which is a Jacobi diagram of $\ideg=p+q+r+2$.

\begin{example}
\label{ex:deg3}
Let us see examples of Theorem 3.14 in a small degree.
For $i,j,k \in \{1^\pm,\dots,g^\pm\}$ distinct, we have
\begin{align*}
 2\lss(T(i_1,j_1,k_1,j_2,i_2)) &= -\ss( T(i,j,k,k,j,i)+2T(i,i,j,k,j,i)+2T(i,j,j,k,j,i) \\
 &\qquad +O(k,j,i,j)+O(k,i,j,i)), \\
 2\lss(O(i_1,j_1,i_2)) &= -\ss(O(i,j,j,i) +2O(i,i,j,i) +\theta(;i;j)),
\end{align*}
which are the case $m=2$ in Corollary~\ref{cor:2torsion} below.
\end{example}

Setting $J=O(a_1,a_2,\ldots,a_m,\ldots, a_3, a_2)$ in Theorem~\ref{thm:symrelation}(\ref{item:even}) and considering the line symmetry $r$ which fixes the univalent vertices labeled by $a_1$ and $a_m$,
we obtain the following.
It is used in the proof of Theorem~\ref{thm:Ker_sn1} in Section~\ref{sec:Ker_OneLoop}.

\begin{corollary}
\label{cor:1LoopRel}
For $m\ge2$ and $a_1, \dots, a_m \in \{1^\pm, \dots, g^\pm\}$,
\begin{align*}
&O(a_1,\ldots,a_{m-1},a_m,a_{m-1},\ldots,a_1)
+O(a_m,\ldots,a_{2},a_1,a_{2},\ldots,a_m)\\
&+\sum_{i=2}^{m-1}\theta(a_{i-1},\ldots,a_1,\ldots,a_{i-1};a_i;a_{i+1},\ldots,a_m,\ldots,a_{i+1})\in\Ker \ss.
\end{align*}
\end{corollary}

\begin{remark}
\label{rem:TrueKer}
By Corollary~\ref{cor:1LoopRel}, we have $O(a_1,a_2,a_1)+O(a_2,a_1,a_2) \in \Ker\ss$, which is proved in \cite[Lemma~6.6(1)]{NSS21} in a different way.
Now it is natural to ask whether
\[
O(a_1, \dots, a_{m-1}, a_m, a_{m-1}, \dots, a_1) + O(a_m, \dots, a_2, a_1, a_2, \dots, a_m) \in \Ker \ss
\]
for $m\geq 3$.
We see that the case $m=3$ is true by Corollary~\ref{cor:1LoopRel} and Lemma~\ref{lem:Theta111} below.
\end{remark}

\begin{lemma}
\label{lem:Theta111}
$\theta(a;b;c)=0 \in \A_{5,2}^{c}$.
\end{lemma}

\begin{proof}
By the IHX relation, we have $\theta(a;;c,b) = \theta(a;b;c)+\theta(a,b;;c)$, where we use the notation \eqref{eq:theta} in the case $q=0$.
Here well-definedness of the map $\bu$ defined in Section~\ref{sec:Jacobi_diagrams} implies that $\theta(a;;c,b) = \bu(O(a,b,c)) = \theta(a,b;;c)$, and hence $\theta(a;b;c)=0$.
\end{proof}

The Jacobi diagrams $T(a_1,a_2,\ldots,a_{m+1},\ldots,a_2,a_1)$ and $O(a_1,a_2,\ldots,a_m,\ldots,a_2,a_1)$ are symmetric with respect to the apparent lines.
The next corollary is obtained by applying Theorem~\ref{thm:symrelation}(\ref{item:odd}) to these Jacobi diagrams.

\begin{corollary}\label{cor:2torsion}
Let $m\ge2$ and $a_1, \dots, a_m \in \{1^\pm, \dots, g^\pm\}$.
Let $J\in\Z\widetilde{\J}_{2m-1}^c$ and $J'\in\Z\widetilde{\J}_{2m-1}^c$ be good lifts respectively of the Jacobi diagrams $T(a_1,a_2,\ldots,a_{m+1},\ldots,a_2,a_1)$ and $O(a_1,a_2,\ldots,a_m,\ldots,a_2,a_1)$ with respect to the line symmetries.
Then, we have
\begin{align*}
2\lss(J)&=-\ss(T(a_1,a_2,\ldots,a_{m+1},a_{m+1},\ldots,a_2,a_1))\\
&\quad-2\sum_{i=1}^m\ss(T(a_1,a_2,\ldots, a_{i-1},a_i,a_i,a_{i+1}\ldots, a_m,a_{m+1},a_m,\ldots,a_2,a_1))\\
&\quad\pm\ss(O(a_{m+1},a_m,a_{m-1},\ldots,a_2,a_1,a_2,\ldots,a_{m-1},a_m))\\
&\quad+\sum_{i=2}^{m}\ss(\pm O(a_{m+1},a_m,a_{m-1},\ldots,a_{i+1},v_{i-1},a_i,v_{i-1},a_{i+1},\ldots,a_{m-1},a_m)),
\displaybreak[1]\\
2\lss(J')&=-\ss(O(a_1,a_2,\ldots,a_m,a_m,\ldots,a_2,a_1))\\
&\quad-2\sum_{i=1}^{m-1}\ss(O(a_1,a_2,\ldots, a_{i-1},a_i,a_i,a_{i+1}, \ldots,a_m,\ldots,a_2,a_1))\\
&\quad +\sum_{i=1}^{m-1}\ss(\pm\theta(a_{i-1},\ldots,a_1,a_1\ldots,a_{i-1};a_i;a_{i+1},\ldots,a_m,\ldots,a_{i+1}))\in Y_{2m-1}\I\C/Y_{2m+1}.
\end{align*}
Here, $O(a_{m+1},a_m,a_{m-1},\ldots,a_{i+1},v_{i-1},a_i,v_{i-1},a_{i+1},\ldots,a_{m-1},a_m)$ denotes the Jacobi diagram obtained by attaching $O(a_{m+1},a_m,a_{m-1},\ldots,a_{i+1},*,a_i,*,a_{i+1},\ldots,a_{m-1},a_m)$ to two copies of $v_{i-1}=T(*,a_{i-1},\ldots, a_2,a_1)$ at the vertices labeled by $*$,
and the signs $\pm$ depend on the choice of a good lift.
\end{corollary}

Corollary~\ref{cor:2torsion} implies that most of the images of symmetric Jacobi diagrams of $\ideg=2m-1$ with $0$-loop or $1$-loop under $\ss$ do not lift to torsion elements in $Y_{2m-1}\I\C/Y_{2m+1}$.
Moreover, in the case $m=2$, we prove that the abelian group $Y_3\I\C/Y_5$ is torsion-free.

In the following proof, we use Proposition~\ref{prop:A4}: $\A_4^c \cong Y_4\I\C/Y_5$ (to be shown later).

\begin{proof}[Proof of Theorem~\textup{\ref{thm:Y3C/Y5}}]
Consider the exact sequence of abelian groups
\[
0 \to Y_4\I\C/Y_5 \to Y_3\I\C/Y_5 \to Y_3\I\C/Y_4 \to 0.
\]
By \cite[Theorem~1.7]{NSS21}, we have $\tor(Y_3\I\C/Y_4) \cong (L_3 \oplus S^2(H))\otimes \Z/2\Z$, where $L_n$ denotes the degree $n$ part of the free Lie algebra on $H$.
Thus, the image of the set
\[
X=
\{T(a,b,c,b,a), T(b,c,a,c,b), T(a,b,b,b,a) \mid a \prec b \prec c\}
\cup
\{O(a',b',a') \mid a'\preceq b'\}
\]
under the map $\ss$ is a basis of $\tor(Y_3\I\C/Y_4)$ over $\Z/2\Z$, where $\prec$ is a total order on the set $\{1^\pm,\dots,g^\pm\}$.
For each $x \in X$, Example~\ref{ex:deg3} (or Corollary~\ref{cor:2torsion}) gives an element $x' \in \A_4^c$ satisfying $\ss(x') = 2\tilde{\ss}(\tilde{x}) \in Y_3\I\C/Y_5$, where $\tilde{x} \in \widetilde{\J}^c_3$ is a good lift of $x$.
If the set $\{x' \mid x \in X\}$ extends to a basis of $\A_4^c$, then $Y_3\I\C/Y_5$ is torsion-free.
By focusing on 0-loop Jacobi diagrams with three pairs of identical labels in Corollary~\ref{cor:2torsion}, it suffices to see that the set
\[
\{T(a,b,c,c,b,a), T(b,c,a,a,c,b) \mid a \prec b \prec c\} \cup \{O(a,b,b,b), O(a',b',b',a') \mid a \prec b,\ a'\preceq b'\}
\]
extends to a basis of $\A_4^c = \bigoplus_{l=0}^3 \A_{4,l}^c$.
This is shown by Example~\ref{ex:A40} and \cite[Proposition~5.2]{NSS21}.
\end{proof}

\begin{remark}
\label{rem:Y3C/Y5}
Theorem~\ref{thm:Y3C/Y5} and the above exact sequence completely determine $Y_3\I\C/Y_5$.
Indeed, $Y_4\I\C/Y_5$ and $(Y_3\I\C/Y_4)/{\tor(Y_3\I\C/Y_4)}$ are computed in Proposition~\ref{prop:A4} and \cite[Theorem~1.7]{NSS21}.

Moreover, once $Y_5\I\C/Y_6$ is determined, we accomplish the determination of the abelian group $Y_3\I\C/Y_6$.
To see it, we use the exact sequence
\begin{align*}
 0 \to Y_5\I\C/Y_6 \to Y_3\I\C/Y_6 \to Y_3\I\C/Y_5 \to 0.
\end{align*}
Since this sequence splits by Theorem~\ref{thm:Y3C/Y5}, one can determine $Y_3\I\C/Y_6$.
In particular, the inclusion induces an isomorphism $\tor(Y_5\I\C/Y_6) \xrightarrow{\cong} \tor(Y_3\I\C/Y_6)$.
\end{remark}

We end this section by discussing the homology cobordism group $\I\H$ of homology cylinders (see \cite[Section~2.1]{NSS21} for example).
Recall that there is a natural projection $\I\C \twoheadrightarrow \I\H$, and let $Y_n\I\H$ denote the image of $Y_n\I\C$.
Also, $[M], [N] \in \I\H$ are said to be \emph{$Y_n$-equivalent} if there is a sequence $M = M_1, M_2, \dots, M_r = N$ in $\I\C$ such that $M_i$ and $M_{i+1}$ are $Y_n$-equivalent or homology cobordant for $i=1,2,\dots,r-1$.
Then we obtain an analogous sequence
\begin{align}
0 \to Y_4\I\H/Y_5 \to Y_3\I\H/Y_5 \to Y_3\I\H/Y_4 \to 0 \label{eq:Y3H/Y5}
\end{align}
to the sequence in the proof of Theorem~\ref{thm:Y3C/Y5}.
However, this is not necessarily exact at the middle since the inclusion $Y_n\I\H \subset \{[M] \in \I\H \mid [M]\sim_{Y_n}[\Sigma_{g,1}\times[-1,1]]\}$ might be proper.

\begin{theorem}
\label{thm:Y3H/Y5}
The sequence \eqref{eq:Y3H/Y5} is exact and the module $Y_3\I\H/Y_5$ is torsion-free.
\end{theorem}

\begin{proof}
In \cite[Corollary~51]{CST16}, it is shown that $\Ker(\ss\colon \A_{2n+1,0}^c \to Y_{2n+1}\I\H/Y_{2n+2})$ coincides with $\Im(\Delta_{n,0}\colon \A_{n,0}^c \to \A_{2n+1,0}^c)$ when $n$ is odd.
Here the map $\Delta_{n,0}$ is defined in \cite[Definition~3.4]{NSS21}, which is denoted by $\Delta_{2n+1}$ in \cite[Definition~4.3]{CST12W}.
Also, $\ss\colon \A_{2n,0}^c \to Y_{2n}\I\H/Y_{2n+1}$ is an isomorphism for all $n$ (\cite[Corollary~50]{CST16}).
Then, we have a commutative diagram
\[
\xymatrix{
0 \ar[r] & \A_{4}^c \ar[r]^-{\ss} \ar@{->>}[d] & Y_3\I\C/Y_5 \ar[r] \ar@{->>}[d] & Y_3\I\C/Y_4 \ar[r] \ar@{<-_)}[d]_-{\ss} & 0 \\
0 \ar[r] & \A_{4,0}^c \ar[r]^-{\ss} \ar[d]^-{\ss}_-{\cong} & \frac{Y_3\I\C/Y_5}{\ss(\A_{4,\geq 1}^c)+\tilde{\ss}(\Z\widetilde{\J}_{3,\geq 1}^c)} \ar@{-->}[r] \ar@{-->>}[d] & \A_{3,0}^c/\Im\Delta_{1,0} \ar[r] \ar[d]^-{\ss}_-{\cong} & 0 \\
0 \ar[r] & Y_4\I\H/Y_5 \ar[r] & Y_3\I\H/Y_5 \ar[r] & Y_3\I\H/Y_4 \ar[r] & 0.
}\]
Here the vertical dashed arrow is induced by \cite[Theorem~2]{Lev01}.
The top and middle rows are exact by Proposition~\ref{prop:A4} and diagram chasing, respectively.
Then the bottom row is also exact.

Now, the latter half of the statement is proved almost in the same way as Theorem~\ref{thm:Y3C/Y5}.
\end{proof}

\section{Structures of modules of Jacobi diagrams}
\label{sec:Jacobi_diagrams}

\subsection{Maps $\bu$ and $\bd$ between Jacobi diagrams}
\label{subsec:bu_bd}

We introduce some maps between modules of Jacobi diagrams, which enable us to understand well the value of the invariant $\bar{z}_{2m-1}$ for $\ss(O(a_1, \dots, a_m, \dots, a_1))$ in Section~\ref{sec:Ker_OneLoop}.
Also, we prove Theorem~\ref{thm:bu_isom} in this subsection.

\begin{definition}[\cite{CDM12}]
The map $\bu \colon \A_{n,l}^c \to \A_{n+2,l+1}^c$ is defined by blowing up a trivalent vertex, that is, by replacing a trivalent vertex of a Jacobi diagram with a loop as in Figure~\ref{fig:bu}.
\end{definition}

The map $\bu$ is well-defined, namely independent of the choice of a trivalent vertex.
Indeed, since $J \in \A_n^c$ is connected, it suffices to compare the results of blowing up at adjacent trivalent vertices.
One can directly check it by the AS and IHX relations.

The map $\bu$ is called the insertion of a triangle into a vertex in \cite[Section~7.2.2]{CDM12} (only for trivalent graphs).
Also, in \cite[Remark~7.10]{CDM12}, it is mentioned that the insertion of a bubble into an edge, which equals $2\bu$, is not necessarily injective.

Recall that $\J_{n}^c$ denotes the set of connected Jacobi diagrams of $\ideg=n$.
Let $\J_{n,2}^c$ be the subset of $\J_{n}^c$ consisting of Jacobi diagrams whose first Betti numbers are two, and let $\R$ be the submodule of the free $\Z$-module $\Z\J_{n,2}^c$ generated by the AS, IHX, and self-loop relators.
Then $\A_{n,2}^c = \Z\J_{n,2}^c/\R$.
Here note that we call an element of $\J_{n,2}^c$ a Jacobi diagram which is not yet an equivalence class.
The \emph{spine} of a Jacobi diagram $J$ is defined to be the graph obtained by collapsing edges incident to univalent vertices until there is no univalent vertex.
When $J \in \J_{n,2}^c$, its spine is either the theta graph or eyeglass graph.
Let $\Theta_n$ be the subset consisting of $J \in \J_{n,2}^c$ whose spine is the theta graph, and let $\R_\Theta$ be the submodule of $\Z\Theta_n$ generated by the AS and IHX relators among diagrams in $\Theta_n$.

Let us define a map $f \colon \Z\J_{n,2}^c/\AS \twoheadrightarrow \Z\Theta_n/\R_\Theta$.
For $J \in \Theta_n$, we simply set $f(J)=J$.
For $J \notin \Theta_n$, under the AS relation, we may assume $J$ is of the form
\begin{align}
\label{eq:eyeglass}
\begin{tikzpicture}[scale=0.3, baseline={(0,0)}, densely dashed]
 \draw (1,0) -- (1,1) node[anchor=south] {$t_1$};
 \node at (2.5,1) {$\cdots$};
 \draw (4,0) -- (4,1) node[anchor=south] {$t_r$};
 \draw (0,0) -- (5,0);
\begin{scope}[xshift=-2cm]
 \draw (0,0) circle [radius=2];
 \draw (135:2) -- (135:3) node[anchor=east] {$\ast$};
 \draw (-135:2) -- (-135:3) node[anchor=east] {$\ast$};
 \node at (-3,0.5) {$\vdots$};
\end{scope}
\begin{scope}[xshift=7cm]
 \draw (0,0) circle [radius=2];
 \draw (45:2) -- (45:3) node[anchor=west] {$\ast$};
 \draw (-45:2) -- (-45:3) node[anchor=west] {$\ast$};
 \node at (3,0.5) {$\vdots$};
\end{scope}
\end{tikzpicture}\ ,
\end{align}
where $t_1,\dots,t_r$ and $\ast$'s are (rooted) trees.
Note that $\ast$'s are not important to define $f$.
In this case, $f(J)$ is defined by
\begin{align}
\label{eq:eyeglass_to_theta}
f(J)=
\sum
\begin{tikzpicture}[scale=0.3, baseline={(0,0)}, densely dashed]
 \draw (0,-2) -- (0,2);
 \draw (0,2) -- (5,2);
 \draw (1,2) -- (1,3) node[anchor=south] {$t_{a_1}$};
 \node at (2.5,3) {$\cdots$};
 \draw (4,2) -- (4,3) node[anchor=south] {$t_{a_p}$};
 \draw (0,-2) -- (5,-2);
 \draw (1,-2) -- (1,-1) node[anchor=south] {$t_{b_1}$};
 \node at (2.5,-1) {$\cdots$};
 \draw (4,-2) -- (4,-1) node[anchor=south] {$t_{b_q}$};
\begin{scope}[xshift=0cm]
 \draw (0,2) arc (90:270:2);
 \draw (135:2) -- (135:3) node[anchor=east] {$\ast$};
 \draw (-135:2) -- (-135:3) node[anchor=east] {$\ast$};
 \node at (-3,0.5) {$\vdots$};
\end{scope}
\begin{scope}[xshift=5cm]
 \draw (0,2) arc (90:-90:2);
 \draw (45:2) -- (45:3) node[anchor=west] {$\ast$};
 \draw (-45:2) -- (-45:3) node[anchor=west] {$\ast$};
 \node at (3,0.5) {$\vdots$};
\end{scope}
\end{tikzpicture}
-\sum
\begin{tikzpicture}[scale=0.3, baseline={(0,0)}, densely dashed]
 \draw (-2,-2) -- (0,2);
 \draw (-2,2) -- (0,-2);
 \draw (0,-2) -- (0,2);
 \draw (0,2) -- (5,2);
 \draw (1,2) -- (1,3) node[anchor=south] {$t_{a_1}$};
 \node at (2.5,3) {$\cdots$};
 \draw (4,2) -- (4,3) node[anchor=south] {$t_{a_p}$};
 \draw (0,-2) -- (5,-2);
 \draw (1,-2) -- (1,-1) node[anchor=south] {$t_{b_1}$};
 \node at (2.5,-1) {$\cdots$};
 \draw (4,-2) -- (4,-1) node[anchor=south] {$t_{b_q}$};
\begin{scope}[xshift=-2cm]
 \draw (0,2) arc (90:270:2);
 \draw (135:2) -- (135:3) node[anchor=east] {$\ast$};
 \draw (-135:2) -- (-135:3) node[anchor=east] {$\ast$};
 \node at (-3,0.5) {$\vdots$};
\end{scope}
\begin{scope}[xshift=5cm]
 \draw (0,2) arc (90:-90:2);
 \draw (45:2) -- (45:3) node[anchor=west] {$\ast$};
 \draw (-45:2) -- (-45:3) node[anchor=west] {$\ast$};
 \node at (3,0.5) {$\vdots$};
\end{scope}
\end{tikzpicture}\ ,
\end{align}
where the sums are taken over all $(p,q)$-shuffles for $p,q \geq 0$ with $p+q=r$, that is, $a_1<\dots<a_p$ and $b_1<\dots<b_q$.
Note that each summation has $2^r$ terms.

\begin{proposition}\label{prop:eyeglass_to_theta}
$f$ is well-defined and induces an isomorphism $\A_{n,2}^c \to \Z\Theta_n/\R_\Theta$, namely all relations among $2$-loop Jacobi diagrams arise from the theta graph.
Moreover, the inverse is the map induced by the inclusion $\Z\Theta_n \hookrightarrow \Z\J_{n,2}^{c}$.
\end{proposition}

\begin{proof}
Let $J \in \J_{n,2}^c$ be a Jacobi diagram of the form \eqref{eq:eyeglass}.
By a 180 degree rotation and the AS relation, $J$ has another expression $\check{J}$ as \eqref{eq:eyeglass}, with sign $(-1)^r$.
Thus, we have to check $f(J)=f((-1)^r\check{J})$.
Applying a 180 degree rotation and the AS relation to each term of $f((-1)^r\check{J})$, we obtain
\begin{align}
\label{eq:left_ver}
\sum
\begin{tikzpicture}[scale=0.3, baseline={(0,0)}, densely dashed]
 \draw (5,-2) -- (5,2);
 \draw (0,2) -- (5,2);
 \draw (1,2) -- (1,3) node[anchor=south] {$t_{a_1}$};
 \node at (2.5,3) {$\cdots$};
 \draw (4,2) -- (4,3) node[anchor=south] {$t_{a_p}$};
 \draw (0,-2) -- (5,-2);
 \draw (1,-2) -- (1,-1) node[anchor=south] {$t_{b_1}$};
 \node at (2.5,-1) {$\cdots$};
 \draw (4,-2) -- (4,-1) node[anchor=south] {$t_{b_q}$};
\begin{scope}[xshift=0cm]
 \draw (0,2) arc (90:270:2);
 \draw (135:2) -- (135:3) node[anchor=east] {$\ast$};
 \draw (-135:2) -- (-135:3) node[anchor=east] {$\ast$};
 \node at (-3,0.5) {$\vdots$};
\end{scope}
\begin{scope}[xshift=5cm]
 \draw (0,2) arc (90:-90:2);
 \draw (45:2) -- (45:3) node[anchor=west] {$\ast$};
 \draw (-45:2) -- (-45:3) node[anchor=west] {$\ast$};
 \node at (3,0.5) {$\vdots$};
\end{scope}
\end{tikzpicture}
-\sum
\begin{tikzpicture}[scale=0.3, baseline={(0,0)}, densely dashed]
 \draw (5,-2) -- (7,2);
 \draw (5,2) -- (7,-2);
 \draw (5,-2) -- (5,2);
 \draw (0,2) -- (5,2);
 \draw (1,2) -- (1,3) node[anchor=south] {$t_{a_1}$};
 \node at (2.5,3) {$\cdots$};
 \draw (4,2) -- (4,3) node[anchor=south] {$t_{a_p}$};
 \draw (0,-2) -- (5,-2);
 \draw (1,-2) -- (1,-1) node[anchor=south] {$t_{b_1}$};
 \node at (2.5,-1) {$\cdots$};
 \draw (4,-2) -- (4,-1) node[anchor=south] {$t_{b_q}$};
\begin{scope}[xshift=0cm]
 \draw (0,2) arc (90:270:2);
 \draw (135:2) -- (135:3) node[anchor=east] {$\ast$};
 \draw (-135:2) -- (-135:3) node[anchor=east] {$\ast$};
 \node at (-3,0.5) {$\vdots$};
\end{scope}
\begin{scope}[xshift=7cm]
 \draw (0,2) arc (90:-90:2);
 \draw (45:2) -- (45:3) node[anchor=west] {$\ast$};
 \draw (-45:2) -- (-45:3) node[anchor=west] {$\ast$};
 \node at (3,0.5) {$\vdots$};
\end{scope}
\end{tikzpicture}\ ,
\end{align}
where the sums are again taken over all $(p,q)$-shuffles.
Let us show that the first (resp.\ second) sum above is equal to the first (resp.\ second) sum in the definition of $f$.
We discuss only the first sum, and the proof for the second sum is almost the same.
In the following proof, it is convenient to regard the IHX relation as Kirchhoff's law.
Now, by the IHX relation, we have
\begin{align}
\label{eq:left_to_right}
\begin{tikzpicture}[scale=0.3, baseline={(0,0)}, densely dashed]
 \draw (5,-2) -- (5,2);
 \draw (0,2) -- (5,2);
 \draw (1,2) -- (1,3) node[anchor=south] {$t_{a_1}$};
 \node at (2.5,3) {$\cdots$};
 \draw (4,2) -- (4,3) node[anchor=south] {$t_{a_p}$};
 \draw (0,-2) -- (5,-2);
 \draw (1,-2) -- (1,-1) node[anchor=south] {$t_{b_1}$};
 \node at (2.5,-1) {$\cdots$};
 \draw (4,-2) -- (4,-1) node[anchor=south] {$t_{b_q}$};
\begin{scope}[xshift=0cm]
 \draw (0,2) arc (90:270:2);
 \draw (135:2) -- (135:3) node[anchor=east] {$\ast$};
 \draw (-135:2) -- (-135:3) node[anchor=east] {$\ast$};
 \node at (-3,0.5) {$\vdots$};
\end{scope}
\begin{scope}[xshift=5cm]
 \draw (0,2) arc (90:-90:2);
 \draw (45:2) -- (45:3) node[anchor=west] {$\ast$};
 \draw (-45:2) -- (-45:3) node[anchor=west] {$\ast$};
 \node at (3,0.5) {$\vdots$};
\end{scope}
\end{tikzpicture}
=\sum
\begin{tikzpicture}[scale=0.3, baseline={(0,0)}, densely dashed]
 \draw (0,-4) -- (0,4);
 \draw (-2,-4) -- (4,-4);
 \draw (-2,4) -- (4,4);
 \draw (1,4) -- (1,5) node[anchor=south] {$t_{a''_1}$};
 \node at (2.5,5) {$\cdots$};
 \draw (4,4) -- (4,5) node[anchor=south] {$t_{a''_{p''}}$};
 \draw (0,3) -- (1,3) node[anchor=west] {$t_{a'_1}$};
 \node at (1,2.5) {$\vdots$};
 \draw (0,1) -- (1,1) node[anchor=west] {$t_{a'_{p'}}$};
 \draw (0,-1) -- (-1,-1) node[anchor=east] {$t_{b'_{q'}}$};
 \node at (-1,-1.5) {$\vdots$};
 \draw (0,-3) -- (-1,-3) node[anchor=east] {$t_{b'_{1}}$};
 \draw (1,-4) -- (1,-3) node[anchor=south] {$t_{b''_1}$};
 \node at (2.5,-3) {$\cdots$};
 \draw (4,-4) -- (4,-3) node[anchor=south] {$t_{b''_{q''}}$};
\begin{scope}[xshift=-2cm]
 \draw (0,4) arc (90:270:4);
 \draw (135:4) -- (135:5) node[anchor=east] {$\ast$};
 \draw (-135:4) -- (-135:5) node[anchor=east] {$\ast$};
 \node at (-5,0.5) {$\vdots$};
\end{scope}
\begin{scope}[xshift=4cm]
 \draw (0,4) arc (90:-90:4);
 \draw (45:4) -- (45:5) node[anchor=west] {$\ast$};
 \draw (-45:4) -- (-45:5) node[anchor=west] {$\ast$};
 \node at (5,0.5) {$\vdots$};
\end{scope}
\end{tikzpicture}\ ,
\end{align}
where the sum is taken over all $(p',p'')$-shuffles and $(q',q'')$-shuffles for $p'+p''=p$ and $q'+q''=q$.
The terms with $p'=q'=0$ appear in the definition of $f(J)$ as well.
Therefore, one should check that the rest of the terms cancel with each other when we expand each term of the first sum in \eqref{eq:left_ver} as \eqref{eq:left_to_right}.
To see it, let $m = \max\{a'_1,\dots,a'_{p'},b'_1,\dots,b'_{q'}\}$ and focus on the tree $t_m$.
Each Jacobi diagram $J'$ in the summation pairs with another one which is identical with $J'$ except that the edge labeled by $t_m$ is attached on the opposite side of the vertical edge of the spine.
Such pairs are canceled by the AS relation.

In the rest of the proof, we show that $f$ preserves the IHX and self-loop relations.
Any Jacobi diagram $J \in \J_{n,2}^c$ with self-loop is, under the AS relation, of the from \eqref{eq:eyeglass} such that the left or right circle has a single trivalent vertex.
Then $f(J)=0$ by the definition and well-definedness of $f$.
When the IHX relation is applied except the central edge of the eyeglass graph, it is obviously preserved by $f$.
Therefore, what we should consider are the following three cases:
\begin{align*}
r_1 &=
\begin{tikzpicture}[scale=0.3, baseline={(0,-0.1)}, densely dashed]
 \draw (0,-2) -- (0,2);
 \draw (0,0) circle [radius=2];
 \draw (135:2) -- (135:3) node[anchor=east] {$\ast$};
 \draw (-135:2) -- (-135:3) node[anchor=east] {$\ast$};
 \node at (-3,0.5) {$\vdots$};
  \draw (45:2) -- (45:3) node[anchor=west] {$\ast$};
 \draw (-45:2) -- (-45:3) node[anchor=west] {$\ast$};
 \node at (3,0.5) {$\vdots$};
\end{tikzpicture}
-
\begin{tikzpicture}[scale=0.3, baseline={(0,-0.1)}, densely dashed]
 \draw (0,0) -- (2,0);
\begin{scope}[xshift=-2cm]
 \draw (0,0) circle [radius=2];
 \draw (135:2) -- (135:3) node[anchor=east] {$\ast$};
 \draw (-135:2) -- (-135:3) node[anchor=east] {$\ast$};
 \node at (-3,0.5) {$\vdots$};
\end{scope}
\begin{scope}[xshift=4cm]
 \draw (0,0) circle [radius=2];
 \draw (45:2) -- (45:3) node[anchor=west] {$\ast$};
 \draw (-45:2) -- (-45:3) node[anchor=west] {$\ast$};
 \node at (3,0.5) {$\vdots$};
\end{scope}
\end{tikzpicture}
+
\begin{tikzpicture}[scale=0.3, baseline={(0,-0.1)}, densely dashed]
 \draw (-1,2) -- (1,-2);
 \draw (-1,-2) -- (1,2);
 \draw (-0.5,-1) -- (0.5,-1);
\begin{scope}[xshift=-1cm]
 \draw (0,2) arc (90:270:2);
 \draw (135:2) -- (135:3) node[anchor=east] {$\ast$};
 \draw (-135:2) -- (-135:3) node[anchor=east] {$\ast$};
 \node at (-3,0.5) {$\vdots$};
\end{scope}
\begin{scope}[xshift=1cm]
 \draw (0,2) arc (90:-90:2);
 \draw (45:2) -- (45:3) node[anchor=west] {$\ast$};
 \draw (-45:2) -- (-45:3) node[anchor=west] {$\ast$};
 \node at (3,0.5) {$\vdots$};
\end{scope}
\end{tikzpicture}\ ,
\\
r_2 &=
\begin{tikzpicture}[scale=0.3, baseline={(0,-0.1)}, densely dashed]
 \node at (-1,0) {$\cdots$};
 \draw (1,0) -- (1,1) node[anchor=south] {$t$};
 \draw (0,0) -- (2,0);
\begin{scope}[xshift=4cm]
 \draw (0,0) circle [radius=2];
 \draw (45:2) -- (45:3) node[anchor=west] {$\ast$};
 \draw (-45:2) -- (-45:3) node[anchor=west] {$\ast$};
 \node at (3,0.5) {$\vdots$};
\end{scope}
\end{tikzpicture}
-
\begin{tikzpicture}[scale=0.3, baseline={(0,-0.1)}, densely dashed]
 \node at (-1,0) {$\cdots$};
 \draw (0,0) -- (2,0);
\begin{scope}[xshift=4cm]
 \draw (0,0) circle [radius=2];
 \draw (135:2) -- (135:3) node[anchor=east] {$t$};
 \draw (45:2) -- (45:3) node[anchor=west] {$\ast$};
 \draw (-45:2) -- (-45:3) node[anchor=west] {$\ast$};
 \node at (3,0.5) {$\vdots$};
\end{scope}
\end{tikzpicture}
+
\begin{tikzpicture}[scale=0.3, baseline={(0,-0.1)}, densely dashed]
 \node at (-1,0) {$\cdots$};
 \draw (0,0) -- (2,0);
\begin{scope}[xshift=4cm]
 \draw (0,0) circle [radius=2];
 \draw (-135:2) -- (-135:3) node[anchor=east] {$t$};
 \draw (45:2) -- (45:3) node[anchor=west] {$\ast$};
 \draw (-45:2) -- (-45:3) node[anchor=west] {$\ast$};
 \node at (3,0.5) {$\vdots$};
\end{scope}
\end{tikzpicture}\ ,
\\
r_3 &=
\begin{tikzpicture}[scale=0.3, baseline={(0,0)}, densely dashed]
 \node at (-3,0) {$\cdots$};
 \draw (0,1) -- (-1,2) node[anchor=south] {$t$};
 \draw (0,0) -- (0,1);
 \draw (0,1) -- (1,2) node[anchor=south] {$t'$};
 \node at (3,0) {$\cdots$};
 \draw (-2,0) -- (2,0);
\end{tikzpicture}
-
\begin{tikzpicture}[scale=0.3, baseline={(0,0)}, densely dashed]
 \node at (-3,0) {$\cdots$};
 \draw (-1,0) -- (-1,1) node[anchor=south] {$t$};
 \draw (1,0) -- (1,1) node[anchor=south] {$t'$};
 \node at (3,0) {$\cdots$};
 \draw (-2,0) -- (2,0);
\end{tikzpicture}
+
\begin{tikzpicture}[scale=0.3, baseline={(0,0)}, densely dashed]
 \node at (-3,0) {$\cdots$};
 \draw (-1,0) -- (-1,1) node[anchor=south] {$t'$};
 \draw (1,0) -- (1,1) node[anchor=south] {$t$};
 \node at (3,0) {$\cdots$};
 \draw (-2,0) -- (2,0);
\end{tikzpicture}\ .
\end{align*}
It follows from \eqref{eq:eyeglass_to_theta} that $f(r_1)=f(r_2)=0$.
To prove $f(r_3)=0$, we separate the terms in $f(r_3)$ into two cases:
(I) either $t$ or $t'$ is in the upper-side and the other is in the lower-side in \eqref{eq:eyeglass_to_theta};
(II) $t$ and $t'$ are in the same side.
All terms in (I) come from the last two terms of $r_3$, and simply cancel pairwise.
The terms in (II) coming from the last two terms of $r_3$ cancel with those coming from the first one by the IHX relation in $\R_\Theta$.

We now have a well-defined homomorphism $f\colon \A_{n,2}^c \to \Z\Theta_n/\R_\Theta$.
Here the inclusion $\Z\Theta_n \hookrightarrow \Z\J_{n,2}^{c}$ induces a map $\Z\Theta_n/\R_\Theta \to \A_{n,2}^c$, which is the inverse of $f$ by the definition of $f$.
\end{proof}

Using Proposition~\ref{prop:eyeglass_to_theta}, let us prove that $\bu\colon \A_{n-2,1}^c \to \A_{n,2}^c/\ang{\Theta_n^{\geq 1}}$ is an isomorphism.
For $J \in \Theta_n$, let $\gamma_1$, $\gamma_2$, $\gamma_3$ be paths in $J$ corresponding to the three edges of the spine of $J$.
We write $\Theta_n^{\geq 1}$ for the subset of $\Theta_n$ consisting of $J$ such that each $\gamma_j$ has at least one vertex except the endpoints.
Note that the submodule $\ang{\Theta_n^{\geq1}}$ of $\A_{n,2}^c$ defined in Section~\ref{sec:Intro} coincides with the submodule generated by $\Theta_n^{\geq 1}$ due to the AS and IHX relations.

\begin{proof}[Proof of Theorem~\textup{\ref{thm:bu_isom}}]
Let us construct the inverse of $\bu\colon \A_{n-2,1}^c \to \A_{n,2}^c/\ang{\Theta_n^{\geq 1}}$ via an isomorphism $\A_{n,2}^c/\ang{\Theta_n^{\geq 1}} \cong \Z\Theta_n/\ang{\R_\Theta \cup \Theta_n^{\geq 1}}$ coming from Proposition~\ref{prop:eyeglass_to_theta}.
We define a homomorphism $\Z\Theta_n/\ang{\R_\Theta \cup \Theta_n^{\geq 1}} \to \A_{n-2,1}^c$ by sending $J=\theta(a_1,\dots,a_p; b_1,\dots,b_q; c_1,\dots,c_r)$ ($p+q+r=n-2$ and $p\geq r\geq q$) to
\[
\begin{cases}
 0 & \text{if $q \neq 0$,} \\
 O(a_1,\dots,a_p,c_r,\dots,c_1) & \text{if $q=0$ and $r\neq 0$,} \\
 2O(a_1,a_2\dots,a_{n-2}) & \text{if $q=r=0$.}
\end{cases}
\]
This map is well-defined since the IHX relation around trivalent vertices of the spine of $J$ is preserved.
Note that when $q=0$ and $p=r$ the relation $J=(-1)^{n}\theta(c_1,\dots,c_p; ; a_1,\dots,a_p)$ is preserved as well.
By definition, this map is the inverse of $\bu$.
\end{proof}

\begin{definition}
Let $n \geq 3$.
Define the map $\bd \colon \A_{n,2}^c/\ang{\Theta_n^{\geq 1}} \to \A_{n-2,1}^c$ to be the isomorphism defined in the proof of Theorem~\ref{thm:bu_isom}.
That is, $\bd(J)$ is obtained by blowing down a circle in $J$ with three vertices.
\end{definition}

The maps $\bu$ and $\bd$ play an important role to investigate $\A_{n,2}^c$ in Section~\ref{sec:Ker_OneLoop}.
Here we give an easy application of Theorem~\ref{thm:bu_isom}, which is used in the proofs of Theorem~\ref{thm:HigherLoop} and Proposition~\ref{prop:A4}.

\begin{lemma}
\label{lem:bu_small}
The map $\bu\colon \A_{n,1}^c \to \A_{n+2,2}^c$ is an isomorphism when $n \leq 3$, and thus $\A_{4,2}^{c} \cong S^2(H)$ and $\A_{5,2}^c \cong \A_{1,0}^{c}$.
\end{lemma}

\begin{proof}
When $n \leq 3$, we have $\ang{\Theta_n^{\geq 1}} = \{0\}$ by Lemma~\ref{lem:Theta111}.
Hence, Theorem~\ref{thm:bu_isom} implies that $\bu\colon\A_{n,1}^c \to \A_{n+2,2}^c$ is an isomorphism.
The latter half of the statement follows from \cite[Proposition~5.1]{NSS21} for instance.
\end{proof}

The rest of this subsection is devoted to giving an interesting observation related with the facts
\begin{align*}
 & \Delta_{n,0}(J) \in \Ker(\A_{2n+1,0}^{c} \xrightarrow{\ss} Y_{2n+1}\I\C/Y_{2n+2} \twoheadrightarrow Y_{2n+1}\I\H/Y_{2n+2}), \\
 & \Delta_{n,0}(J) \in \Ker(\A_{2n+1,0}^{c} \xrightarrow{\ss} Y_{2n+1}\I\C/Y_{2n+2} \xrightarrow{\bar{z}_{2n+2}} \A_{2n+2}^c\otimes\Q/\Z)
\end{align*}
shown respectively in \cite[Lemma~42]{CST16} and \cite[Proposition~3.5]{NSS21}.
Recall that $\I\H$ denotes the homology cobordism group of homology cylinders.

\begin{proposition}
\label{prop:Delta_in_Ker}
Let $J \in \A_{n,0}^c$.
Then $\Delta_{n,0}(J) \in \Ker(\pi\circ\ss_{2n+3,1}\circ\bu)$.
\end{proposition}

Now, it is natural to ask whether $\Delta_{n,0}(J) \in \Ker(\A_{2n+1,0}^{c} \to Y_{2n+1}\I\C/Y_{2n+2})$.
The case $n=1$ is shown in \cite[Lemma~6.6(2)]{NSS21}.
To prove Proposition~\ref{prop:Delta_in_Ker}, we prepare the lemma below.
Let $\eta \colon \A_{n,0}^c \to H\otimes L_{n+1}$ denote a homomorphism used in \cite[Section~5.2]{NSS21} or \cite[Section~1]{CST12L}, and $\iota \colon H\otimes L_{n+1} \hookrightarrow H^{\otimes(n+2)}$ be the homomorphism induced by the natural embedding of the degree $n+1$ part $L_{n+1}$ of the free Lie algebra on $H$.

\begin{lemma}
\label{lem:iota_eta}
For $J \in \A_{n,0}^c$, the element $\iota\circ\eta(J)$ is contained in the submodule
\[
\Ang{a_0\otimes a_1\otimes\dots\otimes a_{n+1} +(-1)^n a_{n+1}\otimes\dots\otimes a_1\otimes a_0 \mid a_i \in \{1^\pm,\dots,g^\pm\}}.
\]
\end{lemma}

\begin{proof}
In this proof, when $v \neq w \in U(J)$, we distinguish $\ell(v)$ and $\ell(w)$ even if $\ell(v) = \ell(w)$.
Let us fix $v \neq w \in U(J)$ and compare terms of the form $\ell(v) \otimes\dots\otimes \ell(w)$ and $\ell(w) \otimes\dots\otimes \ell(v)$ in $\iota\circ\eta(J)$.
By the AS relation, we may assume $J$ is of the form
\[
\begin{tikzpicture}[scale=0.3, baseline={(0,0)}, densely dashed]
 \draw (0,0) -- (-3,0) node[anchor=east] {$\ell(v)$};
 \draw (-2,0) -- (-2,1) node[anchor=south] {$t_1$};
 \node at (0,1) {$\cdots$};
 \draw (2,0) -- (2,1) node[anchor=south] {$t_r$};
 \draw (0,0) -- (3,0) node[anchor=west] {$\ell(w)$};
\end{tikzpicture}
=
(-1)^n
\begin{tikzpicture}[scale=0.3, baseline={(0,0)}, densely dashed]
 \draw (0,0) -- (-3,0) node[anchor=east] {$\ell(w)$};
 \draw (-2,0) -- (-2,1) node[anchor=south] {$\bar{t}_r$}; 
 \node at (0,1) {$\cdots$};
 \draw (2,0) -- (2,1) node[anchor=south] {$\bar{t}_1$};
 \draw (0,0) -- (3,0) node[anchor=west] {$\ell(v)$};
\end{tikzpicture}\ ,
\]
where $t_i$ is a (rooted) tree and $\bar{t}_i$ is its mirror image, that is, the tree obtained by reversing the cyclic orders of $t_i$,
and the bar has a different meaning in Section~\ref{sec:refined}.
Therefore, for each term of the form $\ell(v) \otimes a_1 \otimes \dots \otimes a_{n} \otimes \ell(w)$, there is a term of the form $(-1)^n\ell(w) \otimes a_{n} \otimes \dots \otimes a_1 \otimes \ell(v)$.
\end{proof}

Recall that $\A_{n,l}^{c,s}$ denote the submodule of $\A_{n,l}^{c}$ generated by symmetric Jacobi diagrams.

\begin{proof}[Proof of Proposition~\textup{\ref{prop:Delta_in_Ker}}]
Under the isomorphism $\ang{B_{2n+3}^s} \cong H^{\otimes(n+2)}\otimes\Z/2\Z$ (\cite[Proposition~5.2]{NSS21}), we have $\bu(\Delta_{n,0}(J)) = \iota\circ\eta(J)$ since asymmetric Jacobi diagrams in $\bu(\Delta_{n,0}(J))$ are canceled each other.
By Lemma~\ref{lem:iota_eta}, $\bu(\Delta_{n,0}(J))$ is written as a sum of elements of the form $O(a_1, \dots, a_{n+2}, \dots, a_1) + O(a_{n+2}, \dots, a_1, \dots, a_{n+2})$.
By Corollary~\ref{cor:1LoopRel} and the definition of $\pi$ in Section~\ref{sec:Intro}, we conclude that $\bu(\Delta_{n,0}(J)) \in \Ker(\pi\circ\ss_{2n+3,1})$.
\end{proof}

\subsection{Goussarov-Habiro conjecture}
In this subsection, we prove the Goussarov-Habiro conjecture for the $Y_5$-equivalence as a corollary of Proposition~\ref{prop:A4}.
We write $L'_n$ for the degree $n$ part of the free quasi-Lie algebra on $H$.
Let $D_n$ (resp.\ $D'_n$) denotes the kernel of the bracket map $H\otimes L_{n+1} \to L_{n+2}$ (resp.\ $H\otimes L'_{n+1} \to L'_{n+2}$).
By \cite[Theorem~1.1]{CST12L}, we have an isomorphism $\A_{n,0}^c \cong D'_n$, and it is torsion-free if $n$ is even.
Since the natural projection $L'_{n} \twoheadrightarrow L_{n}$ is an isomorphism over $\Q$, one conclude $\rank D'_n = \rank D_n$, which is computed by Witt's formula for $\rank L_n$ (see \cite[Theorem~5.11]{MKS04} for example).

\begin{proposition}
\label{prop:A4}
The module $\A_4^c$ is isomorphic to the direct sum of $D_4'$, $\A_{4,1}^c$, $S^2(H)$, and $\Z$.
In particular, $\A_4^c$ is torsion-free, and thus $\ss\colon \A_4^c \to Y_4\I\C/Y_5$ is an isomorphism.
\end{proposition}

\begin{proof}
Let us investigate each direct summand of $\A_4^c = \bigoplus_{l=0}^3 \A_{4,l}^c$.
First, $\A_{n,0}^c \cong D'_n$ by \cite[Theorem~1.1]{CST12L}.
The 1-loop part $\A_{4,1}^c$ is determined in \cite[Proposition~5.2]{NSS21}, which is torsion-free (and its rank is given in Remark~\ref{rem:A41}).
It follows from Lemma~\ref{lem:bu_small} that $\A_{4,2}^{c} \cong \A_{2,1}^{c} \cong S^2(H)$.
Also, one can check that $\A_{4,3}^{c} \cong \Z$ (see \cite[Table~7.1]{CDM12} in the case of rational coefficients).

We then conclude that $\A_4^c$ is torsion-free.
Here recall the facts that $\ss\colon \A_4^c \to Y_4\I\C/Y_5$ is surjective and that it is an isomorphism over $\Q$.
These imply that $\ss\colon \A_4^c \to Y_4\I\C/Y_5$ is an isomorphism.
\end{proof}

\begin{remark}
\label{rem:A41}
By \cite[Proposition~5.2]{NSS21}, the rank of $\A_{4,1}^{c}$ is computed as follows:
\[
\frac{1}{2}\frac{1}{4}\left( \varphi(1)(2g)^4+\varphi(2)(2g)^2+\varphi(4)2g \right)+\frac{1}{4}(2g+1)(2g)^2 = 2g^4+2g^3+\frac{3}{2}g^2+\frac{1}{2}g,
\]
where $\varphi$ is Euler's totient function.
\end{remark}

As mentioned in Section~\ref{sec:Intro}, the Goussarov-Habiro conjecture is a fundamental question about a relation between the $Y_n$-equivalence and finite type invariants of homology cylinders.
For the definition of finite type invariant, we refer the reader to \cite{Gou99}, \cite{Hab00C}, and \cite[Section~2.3]{MaMe13}.

The degree 4 part of the $Y$-reduction of the LMO functor $\Ztilde$ gives a map $\Ztilde_4^Y\colon \I\C \to \A_4^Y\otimes\Q$, which is a finite type invariant of degree at most 4 (\cite[Theorem~7.11]{CHM08}), where $\A_n^Y$ is the submodule of $\A_n$ generated by Jacobi diagrams without strut.
Moreover, it induces a homomorphism $\Ztilde_4^Y\colon Y_4\I\C/Y_5 \to \A_4^c \subset \A_4^Y\otimes\Q$ satisfying $\Ztilde_4^Y \circ \ss_4 = \id_{\A_4^c}$ up to sign.
The next corollary is proved in much the same way as \cite[Proposition~6.12]{NSS21} and \cite[Section~5.1]{MaMe13}.

\begin{corollary}
\label{cor:GHC}
For $M, M' \in \I\C$, they are $Y_{5}$-equivalent if and only if $f(M)=f(M')$ for every finite type invariant $f$ of degree at most $4$.
\end{corollary}

\begin{proof}
If $M \sim_{Y_5} M'$, then they are not distinguished by finite type invariants of degree at most $4$ (see \cite[Lemma~2.3]{MaMe13} for instance).

Conversely, suppose that $f(M)=f(M')$ for every finite type invariant $f$ of degree at most $4$.
Then we see $M \sim_{Y_4} M'$ by \cite[Proposition~6.12]{NSS21}, and $\Ztilde^Y_{4}(M) = \Ztilde^Y_{4}(M') \in \A_4^Y\otimes\Q$.
Since $Y_4\I\C/Y_5$ is a group (\cite[Theorem~3]{Gou99}, \cite[Section~8.5]{Hab00C}), there is $N \in Y_4\I\C$ such that $M \sim_{Y_5} N\circ M'$.
Thus, we have $\Ztilde^Y_0(N)=\emptyset$ and $\Ztilde^Y_i(N)=0$ for $i=1,2,3$.
Now, by the formula
\[
\Ztilde^Y_4(M) = \sum_{i=0}^{4} \Ztilde^Y_i(N)\star\Ztilde^Y_{4-i}(M')
\]
(see \cite[Corollary~8.3]{CHM08} for the details), we conclude $\Ztilde^Y_4(N)=0 \in \A_4^Y\otimes\Q$.
Since $\Ztilde^Y_4\colon Y_4\I\C/Y_5 \xrightarrow{\cong} \A_4^c \hookrightarrow \A_4^Y\otimes\Q$ is injective by Proposition~\ref{prop:A4}, $N=0 \in Y_4\I\C/Y_5$, and hence $M$ is $Y_{5}$-equivalent to $M'$.
\end{proof}

\section{Kernel of $\ss$ restricted to $\A_{n,1}^c$}
\label{sec:Ker_OneLoop}
The module structure of $\A_{n,1}^c$ is determined in \cite[Proposition~5.2]{NSS21}.
Recall that $\ang{B_n^s}$ is the submodule of $\A_{n,1}^c$ generated by the set $B_n^s$ written after Remark~\ref{rem:symmetric}.
To obtain a more precise description of $\ang{B_n^s}$, we introduce a necklace with arrow, and complete the proof of Theorem~\ref{thm:Ker_sn1}.
Here, a necklace consists of beads which are elements of the set $\{1^\pm,\dots,g^\pm\}$ and a necklace of length $n$ is considered up to the action of the dihedral group $\D_{2n}$.

\begin{definition}
A \emph{necklace with arrow} is a pair of a symmetric necklace and an arrow which is an axis of symmetry.
An arrow points to either a bead or the midpoint of adjacent beads.
The set $\N_{2m}$ of necklaces with arrow is separated into subsets $\N_{2m}'$ and $\N_{2m}''$, where $\N_{2m}'$ (resp.\ $\N_{2m}''$) denotes the set of necklaces with arrow pointing to a single bead (resp.\ the midpoint of adjacent beads) as illustrated on the left (resp.\ right) in Figure~\ref{fig:Necklace}.
The right one is denoted by $O(a_1,\dots,a_m \uparrow a_m,\dots,a_1)$.
\end{definition}

\begin{figure}[h]
 \centering
\begin{tikzpicture}[scale=0.3, baseline={(0,0)}, ]
 \node at (180:3) {$a_1$};
 \node at (140:3) {$a_2$};
 \node at (-140:3) {$a_2$};
 \node at (40:3) {$a_m$};
 \node at (-40:3) {$a_m$};
 \node at (0:3.5) {$a_{m+1}$};
 \draw [->] (-2,0) -- (2,0);
 \draw [dotted, thick] (80:3) arc (80:100:3) (100:3);
 \draw [dotted, thick] (-80:3) arc (-80:-100:3) (-100:3);
\end{tikzpicture}
\quad
\begin{tikzpicture}[scale=0.3, baseline={(0,0)}, ]
 \node at (160:3) {$a_1$};
 \node at (-160:3) {$a_1$};
 \node at (120:3) {$a_2$};
 \node at (-120:3) {$a_2$};
 \node at (20:3) {$a_m$};
 \node at (-20:3) {$a_m$};
 \draw [->] (-2,0) -- (2,0);
 \draw [dotted, thick] (60:3) arc (60:80:3) (80:3);
 \draw [dotted, thick] (-60:3) arc (-60:-80:3) (-80:3);
\end{tikzpicture}
 \caption{Two types of necklaces with arrow.}
 \label{fig:Necklace}
\end{figure}
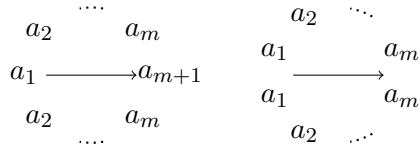

Note that considering a necklace with arrow is equivalent to fixing a base point on a necklace.
The former is more convenient to discuss $\ang{B_n^s}$.
Forgetting arrows, we obtain a natural map $\N_{2m} \twoheadrightarrow \ang{B_{2m}^s}$ (see Figure~\ref{fig:Jacobi_from_necklace}).

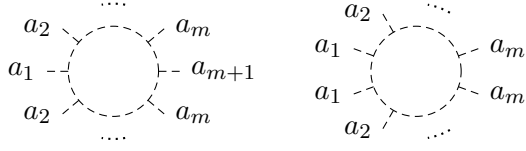
\begin{figure}[h]
 \centering
\begin{tikzpicture}[scale=0.3, baseline={(0,0)}, densely dashed]
 \draw (0,0) circle [radius=2];
 \draw (180:2) -- (180:3) node[anchor=east] {$a_1$};
 \draw (140:2) -- (140:3) node[anchor=east] {$a_2$};
 \draw (-140:2) -- (-140:3) node[anchor=east] {$a_2$};
 \draw (40:2) -- (40:3) node[anchor=west] {$a_m$};
 \draw (-40:2) -- (-40:3) node[anchor=west] {$a_m$};
 \draw (0:2) -- (0:3) node[anchor=west] {$a_{m+1}$};
 \draw [dotted, thick] (80:3) arc (80:100:3) (100:3);
 \draw [dotted, thick] (-80:3) arc (-80:-100:3) (-100:3);
\end{tikzpicture}
\quad
\begin{tikzpicture}[scale=0.3, baseline={(0,0)}, densely dashed]
 \draw (0,0) circle [radius=2];
 \draw (160:2) -- (160:3) node[anchor=east] {$a_1$};
 \draw (-160:2) -- (-160:3) node[anchor=east] {$a_1$};
 \draw (120:2) -- (120:3) node[anchor=east] {$a_2$};
 \draw (-120:2) -- (-120:3) node[anchor=east] {$a_2$};
 \draw (20:2) -- (20:3) node[anchor=west] {$a_m$};
 \draw (-20:2) -- (-20:3) node[anchor=west] {$a_m$};
 \draw [dotted, thick] (60:3) arc (60:80:3) (80:3);
 \draw [dotted, thick] (-60:3) arc (-60:-80:3) (-80:3);
\end{tikzpicture}
 \caption{The Jacobi diagrams obtained from necklaces with arrow in Figure~\ref{fig:Necklace}.}
 \label{fig:Jacobi_from_necklace}
\end{figure}

Also, we define a map $\mh \colon \N_{2m}'' \to \ang{B_{2m-1}^{s}}$ by merging the two adjacent beads close to the head of the arrow, namely
\[
\mh(O(a_1,\dots,a_m \uparrow a_m,\dots,a_1)) = O(a_1,\dots,a_m,\dots,a_1).
\]
Furthermore, this map induces an isomorphism $\mh \colon (\Z/2\Z)\N_{2m}'' \to \ang{B_{2m-1}^s}$ by \cite[Proposition~5.2]{NSS21}.
Similarly, $\mht \colon \N_{2m}'' \to \ang{B_{2m-2}^s}$ is defined by merging the two beads close to the head and merging two beads close to the tail.

For each $x \in \N_{2m}$, there is a unique non-negative integer $e=e(x)$ such that $x=x^\pi=\dots=x^{\pi/2^{e-1}} \neq x^{\pi/2^e}$, where $x^{\pi/2^{k}}$ denotes the necklace with arrow obtained by rotating only the arrow of $x$ by $\pi/2^k$.

\begin{definition}
 Let $\iota \colon \N_{2m} \to \N_{2m}$ be the map defined by $\iota(x) = x^{\pi/2^e}$.
\end{definition}

\begin{example}
For $a \neq b$, let $x = O(a,b,b,a,a,b,b,a \uparrow a,b,b,a,a,b,b,a) \in \N_{16}$.
Then $e(x)=2$, $\iota(x) = O(b,a,a,b,b,a,a,b \uparrow b,a,a,b,b,a,a,b)$, and
\begin{align*}
 \mh(x) &=O(a,b,b,a,a,b,b,a,b,b,a,a,b,b,a), \\
 \mh(\iota(x)) &=O(b,a,a,b,b,a,a,b,a,a,b,b,a,a,b) \in \A_{15,1}^{c,s}.
\end{align*}
\end{example}

\begin{lemma}\label{lem:iota}
The map $\iota$ is an involution without fixed point.
\end{lemma}

\begin{proof}
The map $\iota$ is fixed point free by definition.
Let $y=\iota(x)$ and $e=e(x)$.
Then for $k=0,1,\dots,e-1$ we have $y^{\pi/2^k} = (x^{\pi/2^k})^{\pi/2^e} = x^{\pi/2^e} = y$.
Also, $y^{\pi/2^e} = x^{\pi/2^{e-1}} = x$, and thus, $\iota(y)=x$.
\end{proof}

Since $\#\N_{2m}' = (2g)^{m+1}$ and $\#\N_{2m}'' = (2g)^{m}$, Lemma~\ref{lem:iota} implies that $\#(\N_{2m}/\iota) = \frac{1}{2}((2g)^{m+1}+(2g)^{m})$.
It follows from the proof of \cite[Proposition~5.2]{NSS21} that $\rank_\Z \ang{B_{2m}^s} = \frac{1}{2}(2g+1)(2g)^m$.
Therefore, a natural map $\Z\N_{2m} \twoheadrightarrow \ang{B_{2m}^s}$ induces an isomorphism $\Z(\N_{2m}/\iota) \to \ang{B_{2m}^s}$.
This isomorphism enables us to identify connected symmetric Jacobi diagrams with necklaces with arrow.

We consider the composite map
\[
\A_{2m-1,1}^{c} \xrightarrow{\ss} Y_{2m-1}\I\C/Y_{2m} \xrightarrow{\bar{z}_{2m,2}} \A_{2m,2}^{c}\otimes\Q/\Z
\]
Since the image of this homomorphism is included in $\A_{2m,2}^c\otimes(\frac{1}{2}\Z/\Z)$, one obtains a homomorphism
\[
\A_{2m-1,1}^{c} \xrightarrow{(\bar{z}_{2m,2} \circ \ss)\otimes 2} \A_{2m,2}^{c}\otimes\Z/2\Z.
\]
Moreover, it induces a homomorphism
\[
f\colon \A_{2m-1,1}^{c} \to (\A_{2m,2}^{c}/\ang{\Theta_{2m}^{\geq 1}})\otimes\Z/2\Z \xrightarrow{\bd} \A_{2m-2,1}^{c}\otimes\Z/2\Z,
\]
which sends $O(a_1, a_2, \dots, a_m, \dots, a_2, a_1)$ to $O(a_1, a_2, \dots, a_m, \dots, a_2)$ by \cite[Theorem~1.1]{NSS21}.

Let us prove Theorem~\ref{thm:Ker_sn1}: $\rank_{\Z/2\Z} \Ker(\pi\circ\ss_{2m-1,1}) = \frac{1}{2}((2g)^m-(2g)^{\lceil m/2 \rceil})$ for $m\geq 2$.
Note that $\Ker(\pi\circ\ss_{2m-1,1}) \subset \tor\A_{2m-1,1}^c$.

\begin{proof}[Proof of Theorem~\textup{\ref{thm:Ker_sn1}}]
It follows from Corollary~\ref{cor:1LoopRel} that
\[
\Ang{\mh(x)+\mh(\iota(x)) \in \ang{B_{2m-1}^s} \Bigm| x \in \N_{2m}^{e=0}} \subset \Ker(\pi\circ\ss_{2m-1,1}),
\]
where $\N_{2m}^{e=0} = \{x \in \N''_{2m} \mid e(x)=0\}$.
We prove that this inclusion is actually an equality, and then the proof is complete since the rank of the $\Z/2\Z$-module on the left is $\frac{1}{2}((2g)^m-(2g)^{\lceil m/2 \rceil})$.

Since $\bar{z}_{2m,1}$ factors through $\pi$, we obtain $\Ker(\pi\circ\ss_{2m-1,1}) \subset \Ker(\bar{z}_{2m,1}\circ\ss_{2m-1,1})$.
It follows from \cite[Theorem~1.1]{NSS21} that for $J=O(a_1,\dots,a_{m-1},a_m,a_{m-1},\dots,a_1)$ we have
\begin{align*}
\bar{z}_{2m,1}\circ\ss(J) 
 = \delta'(J)
 = \frac{1}{2}O(a_1,\dots,a_{m-1},a_m,a_m,a_{m-1},\dots,a_1).
\end{align*}
Therefore, one obtains a commutative diagram
\[
\xymatrix{
 (\Z/2\Z)\N_{2m}'' \ar[d]_\mh \ar[r] & (\Z/2\Z)(\N_{2m}/\iota) \ar[d] \\
 \ang{B_{2m-1}^s} \ar[r]^-{\bar{z}_{2m,1}\circ\ss} & \ang{B_{2m}^s}\otimes\Z/2\Z,
}\]
where the vertical arrows are the isomorphisms explained above.
We now conclude that
\[
\Ker(\bar{z}_{2m,1}\circ\ss_{2m-1,1}) \subset \Ang{\mh(x)+\mh(\iota(x)) \Bigm| x, \iota(x) \in \N''_{2m}}.
\]
When $m$ is odd, this completes the proof since $\N_{2m}^{e=0}$ coincides with $\{x \in \N''_{2m} \mid \iota(x) \in \N''_{2m}\}$.

We next discuss the case $m$ even.
Since
$
(\bar{z}_{2m,2}\otimes2)\circ\ss(x) = \delta'(x) \in \ang{\Theta_{2m}^{\geq 1}}
$
for $x \in \ang{\Theta_{2m-1}^{\geq 1,s}}$, the above map $f$ factors through $\pi\circ\ss_{2m-1,1}$ as follows:
\[
\xymatrix{
 \A_{2m-1,1}^c \ar[d]_{s_{2m-1,1}} \ar[drr]^-f & & \\ 
 Y_{2m-1}\I\C/Y_{2m} \ar[r]^-{\bar{z}_{2m,2}\otimes 2} \ar[d]_\pi & (\A_{2m,2}^c/\ang{\Theta_{2m}^{\geq 1}})\otimes\Z/2\Z \ar[r]_-\bd & \A_{2m-2,1}^c\otimes\Z/2\Z, \\
 (Y_{2m-1}\I\C/Y_{2m})/\ss(\ang{\Theta_{2m-1}^{\geq 1,s}}) \ar@{-->}[ur] & & 
}\]
where the module $\ang{\Theta_{2m-1}^{\geq 1,s}}$ is defined in Section~\ref{sec:Intro}.
Hence we have $\Ker(\pi\circ\ss_{2m-1,1}) \subset \Ker f$.
Now, it suffices to show that
\[
\Ang{\mh(x)+\mh(\iota(x)) \Bigm| x \in \N_{2m}^{e\geq 1}} \cap \Ker f = \{0\},
\]
where $\N_{2m}^{e\geq 1} = \{x \in \N''_{2m} \mid e(x)\geq 1,\ \iota(x) \in \N''_{2m}\}$.
Note that any $x \in \N_{2m}^{e\geq 1}$ can be uniquely written as 
\[
x = O(\underbrace{w,\bar{w},\dots,w,\bar{w}}_{2^e} \uparrow \underbrace{w,\bar{w},\dots,w,\bar{w}}_{2^e})
\]
for some word $w \in \{1^\pm,\dots,g^\pm\}^{m/2^e}$ with $w \neq \bar{w}$, where $\bar{w}$ denotes the word obtained from $w$ by reversing the order (see Example~\ref{ex:Oab} below).
By definition, we have 
\[
f(\mh(x) + \mh(\iota(x))) = \mht(x) + \mht(\iota(x)) \in \ang{B_{2m-2}^s}\otimes\Z/2\Z
\]
for $x \in \N_{2m}^{e\geq 1}$.
This element is contained in
\[
\ang{B_{2m-2}^{s,\mathrm{period}}}\otimes\Z/2\Z \cong \ang{B_{m-1}^s} \cong H^{\otimes m/2}\otimes\Z/2\Z.
\]
See \cite[Lemma~5.3]{NSS21} for the definitions of the superscript ``period'' and the first isomorphism, and here recall Remark~\ref{rem:symmetric}.
Under these isomorphisms, the image of $x = O(w,\bar{w},\dots,w,\bar{w} \uparrow w,\bar{w},\dots,w,\bar{w})$ is $w\bar{w}\cdots w\bar{w} + \bar{w}w\cdots \bar{w}w$, where each term consists of $2^{e-1}$ words.
Since these two terms are linearly independent in $H^{\otimes m/2}\otimes\Z/2\Z$, this completes the proof.
\end{proof}

\begin{example}\label{ex:Oab}
For $a \neq b$, let $x = O(a,b,b,a,a,b,b,a \uparrow a,b,b,a,a,b,b,a) \in \N_{16}$.
Then $w=ab$, $\bar{w}=ba$, and $x$ is sent to $abba + baab \in H^{\otimes 4}\otimes\Z/2\Z$.
\end{example}

\section{Kernel of $\ss$ restricted to $\A_{n,l}^c$ for $l>1$}
\label{sec:Ker_HigherLoop}
We introduce a weight system, and prove Theorem~\ref{thm:HigherLoop} which gives lower bounds on the ranks of the $\Z/2\Z$-modules $\Ker(\ss|_{\tor\A_{2k+1,k}^c})$ and $\Im(\ss|_{\tor\A_{2k+1,k}^c})$ for $k\geq 0$.

One usually constructs a weight system from a Lie algebra over $\mathbb{C}$ (see \cite[Section~6.3]{CDM12} for instance).
However, we need to define it over $\Z$.
Let $R$ be a commutative unital ring, and let $\g$ be a free $R$-module with a basis $\{e_i\}_{i=1}^d$.
Suppose $c_{ijk} \in R$ ($i,j,k=1,\dots,d$) satisfy that 
\begin{align}
 & c_{ijk} = -c_{jik}, \label{eq:AS} \\
 & \sum_{m=1}^d (c_{ijm}c_{mkl} - c_{lim}c_{mjk} + c_{ljm}c_{mik}) =0, \label{eq:IHX} \\
 & c_{iik} = 0, \label{eq:self-loop} \\
 & c_{ijk} = c_{jki} = c_{kij}. \label{eq:cyclic}
\end{align}
These conditions arise from a Lie algebra as follows.
Let $\g$ be a Lie algebra over $R$ with basis $\{e_i\}_{i=1}^d$.
(While $R$ is not a field but a ring, we use the terminology ``Lie algebra''.)
Define the structure constants $c_{ijk} \in R$ by $[e_i,e_j] = \sum_{k}c_{ijk}e_k$ and suppose they satisfy $\delta([e_i,e_k],e_j)=\delta(e_i,[e_k,e_j])$ (see \cite[Appendix~A.1]{CDM12} for example).
Here $\delta$ is the bilinear form defined by $\delta(e_i,e_j)=\delta_{ij}$, where $\delta_{ij}$ denotes the Kronecker delta.
Then one can check that the $c_{ijk}$'s satisfy the conditions \eqref{eq:AS}--\eqref{eq:cyclic}.

\begin{example}
Let $\g$ be a semi-simple Lie algebra over $\mathbb{C}$.
Then the structure constants $c_{ijk}$ associated with an orthonormal basis with respect to the Killing form satisfy the conditions \eqref{eq:AS}--\eqref{eq:cyclic}.
\end{example}

Let $H_R = H_1(\Sigma_{g,1};R)$ and identify $H_R$ with the $R$-module $R\{1^\pm,\dots,g^\pm\}$.
For an $R$-module $M$, let $S(M)$ denote the symmetric algebra over $R$.
Recall that $d = \rank_R\g$.

\begin{definition}
The \emph{weight system} associated with $\g$ is a $\Z$-module homomorphism $W_\g \colon \A_n \to S(H_R \otimes_R \g)$ defined as follows.
Let $J \in \A_n$ be a Jacobi diagram, and let $v_1,\dots,v_m$ (resp.\ $w_1,\dots,w_n$) be the univalent (resp.\ trivalent) vertices of $J$.
Put labels $\{1,\dots,d\}$ on the edges of $J$, and let $c(w_l)=c_{ijk}$ if the edges adjacent to $w_l$ are labeled by $i,j,k$ in the cyclic order of $w_l$.
Then, define $W_\g(J) \in S^m(H_R \otimes_R \g)$ by
\[
W_\g(J) = \sum c(w_1)\cdots c(w_n) (l(v_1)\otimes e_{j_1})\cdots(l(v_m)\otimes e_{j_m}),
\]
where $j_i \in \{1,\dots,d\}$ is the label of the edge adjacent to $v_i$, and the sum is taken over all the ways of assigning labels to edges.
Note that $c(w_l)$ depends only on the cyclic order by \eqref{eq:cyclic}.

In other words, for a Jacobi diagram $J$ without strut, we first put $t=\sum_{i,j,k}c_{ijk}e_i \otimes e_j \otimes e_k \in \g^{\otimes 3}$ at each trivalent vertex.
Note that $t$ is invariant under cyclic permutation by \eqref{eq:cyclic}.
Next, using the bilinear form $\delta$, we take the contraction of the form
\[
\g^{\otimes 3}\otimes\g^{\otimes 3} \to \g^{\otimes 4},\ (e_i\otimes e_j\otimes e_k, e_{i'}\otimes e_{j'}\otimes e_{k'}) \mapsto \delta(e_i,e_{i'})e_j\otimes e_k\otimes e_{j'}\otimes e_{k'}
\]
along every edge connecting two trivalent vertices (see \cite[Section~6.2.1]{CDM12} for instance).
Then, consider the tensor product $l(v_i) \otimes e_{j_i}$ at each univalent vertex $v_i$, and take the summation of the products of $\delta_{ij}$'s and $l(v_i) \otimes e_{j_i}$.
\end{definition}

The weight system $W_\g$ is well-defined since one can check that $W_\g$ preserves the AS, IHX, and self-loop relations by \eqref{eq:AS}, \eqref{eq:IHX} and \eqref{eq:self-loop}, respectively.
For example, to prove that $W_\g$ preserves \eqref{eq:IHX}, we consider
\[
\Igraph{i}{j}{l}{k} - \Hgraph{i}{j}{l}{k} +
\begin{tikzpicture}[scale=0.5, baseline={(0,-0.1)}, densely dashed]
  \draw (-1,1) -- (1,-1) node[at start, anchor=south] {$i$} node[anchor=north] {$k$};
  \draw (-1,-1) -- (1,1) node[at start, anchor=north] {$l$} node[anchor=south] {$j$};
  \draw (-0.5,-0.5) -- (0.5,-0.5);
\end{tikzpicture}\ ,
\]
where the central edge in each diagram is labeled by $m$ which runs over $1,\dots,d$.

\begin{example}
\label{ex:sl2Z}
Let $\g=\Ang{e_1,e_2,e_3}_\Z$ and
\[
c_{ijk} = 
\begin{cases}
 \sgn(ijk) & \text{if $\#\{i,j,k\}=3$,} \\
 0 & \text{if $\#\{i,j,k\}<3$.}
\end{cases}
\]
One can directly check that these $c_{ijk}$'s satisfy \eqref{eq:AS}--\eqref{eq:cyclic}.
Also, it arises from $\mathfrak{sl}(2,\mathbb{C})$ with basis $\{(E+F)/2\sqrt{-1}, H/2\sqrt{-1}, (E-F)/2\}$, where
\[
E=
\begin{pmatrix}
 0 & 1 \\
 0 & 0
\end{pmatrix},\
H=
\begin{pmatrix}
 1 & 0 \\
 0 & -1
\end{pmatrix},\
F=
\begin{pmatrix}
 0 & 0 \\
 1 & 0
\end{pmatrix}.
\]
Here we give an example of computation.
For $J=O(a,b) \in \A_{2,1}^c$, we have
\[
W_\g(O(a,b)) = -2(a\otimes e_1)(b\otimes e_1)-2(a\otimes e_2)(b\otimes e_2)-2(a\otimes e_3)(b\otimes e_3).
\]
Note that the computation of this weight system corresponds to counting edge-colorings (with sign) by three colors $\{1,2,3\}$ such that all the three colors appear around each trivalent vertex.
\end{example}

\begin{example}
\label{ex:A40}
The weight system in Example~\ref{ex:sl2Z} enables us show that the set $\{T(a,b,c,c,b,a), T(b,c,a,a,c,b)\}$ extends to a basis of $\A_{4,0}^c$ when $a,b,c \in \{1^\pm,\dots,g^\pm\}$ are mutually distinct.
Indeed, one can check that the coefficients of $(a\otimes e_1)^2(b\otimes e_2)^2(c\otimes e_2)^2$ and $(a\otimes e_1)^2(b\otimes e_1)^2(c\otimes e_2)^2$ in $W_\g(T(a,b,c,c,b,a)) \in S^6(H\otimes\g)$ are 1 and 0, respectively.
On the other hand, the coefficients in $W_\g(T(b,c,a,a,c,b))$ are 0 and 1, respectively.
\end{example}

In the rest of this section, we consider the weight system in Example~\ref{ex:sl2Z}, and simply write $W$.
By Lemma~\ref{lem:Always2} below, it induces a homomorphism
\[
W_m\colon \A_n\otimes\Q/\Z \to S(H \otimes \Z e_m)\otimes\Q/\Z,\ J \mapsto \frac{1}{2}\pr_m \circ W(J),
\]
where $\pr_m\colon S(H \otimes \g) \to S(H \otimes \Z e_m)$ denotes the projection.

\begin{lemma}
\label{lem:Always2}
Let $m \in \{1,2,3\}$ and $n\geq 1$.
The image of the composite map $\A_n \xrightarrow{W} S(H\otimes\g) \xrightarrow{\pr_m} S(H \otimes \Z e_m)$ is included in $2S(H \otimes \Z e_m)$.
\end{lemma}

\begin{proof}
We may assume $m=1$, and use the comment at the end of Example~\ref{ex:sl2Z}.
For each edge-coloring of a Jacobi diagram $J \in \A_n$ such that any edge adjacent to a univalent vertex is colored by $1$, there is an edge-coloring obtained by replacing all 2's (resp.\ 3's) with 3 (resp.\ 2).
Therefore, the coefficient of $x\otimes e_1$ is even for $x \in H$.
\end{proof}

We now give a useful relation between the weight system and the map $\bu$, which is essentially written, for example, in \cite[Lemma~6.15 and Remark~6.16]{CDM12}.

\begin{lemma}
\label{lem:Weight-bu}
$W(\bu^{k}(J)) = (-1)^kW(J)$ for any $k \geq 0$, where $\bu^{k}\colon \A_{n,l}^c \to \A_{n+2k,l+k}^c$ denotes the $k$ times composition of $\bu$.
\end{lemma}

\begin{proof}
It suffices to prove the case $k=1$.
By the comment at the end of Example~\ref{ex:sl2Z}, for each edge-coloring of $J$, there is a unique edge-coloring of $\bu(J)$ extending it.
This gives a bijection between the edge-colorings of $J$ and those of $\bu(J)$, with opposite sign.
This implies $W(\bu(J)) = -W(J)$.
\end{proof}

Corollary~\ref{cor:buTaba} and Proposition~\ref{prop:LRzip} below are keys to prove Theorem~\ref{thm:HigherLoop}.

\begin{corollary}
\label{cor:buTaba}
For any $a, b \in \{1^\pm,\dots,g^\pm\}$, 
\[
W_1\circ\bar{z}_{2k+2,k+1}\circ\ss\circ\bu^{k}(T(a,b,a)) = \frac{1}{2}(a\otimes e_1)(b\otimes e_1) \in S(H \otimes \Z e_1)\otimes\Q/\Z.
\]
In particular, $\bar{z}_{2k+2,k+1}$ is non-trivial.
\end{corollary}

\begin{proof}
Let $J=T(a,b,a)$.
By \cite[Theorem~1]{NSS21}, we have
\[
\bar{z}_{2k+2,k+1}\circ\ss\circ\bu^{k}(J) = \frac{1}{2}\delta''(\bu^{k}(J)) = \bu^{k}\left( \frac{1}{2}\delta''(J) \right).
\]
If follows from Lemma~\ref{lem:Weight-bu} and $\delta''(J) = O(a,b)$ that $W_1\circ\bu^{k}(\frac{1}{2}\delta''(J)) = W_1(\frac{1}{2}O(a,b))$.
Finally, this is equal to $\frac{1}{2}(a\otimes e_1)(b\otimes e_1)$ by Example~\ref{ex:sl2Z}.
\end{proof}

\begin{proposition}
\label{prop:LRzip}
$\ss(\bu^{k}(T(a,b,a))) = \ss(\bu^{k}(T(b,a,b))) \in Y_{2k+1}\I\C/Y_{2k+2}$.
\end{proposition}

\begin{proof}
Let $J \in \Z\widetilde{\J}_{2k,k}^c$ be a lift of $\bu^{k-1}(O(a,b))$.
By Corollary~\ref{cor:AS}, we have
\begin{align*}
\lss(J)
&= \lss(\rev(J)) \\
&= \lss(J) + \ss(\bu^{k-1}(O(a,b,a))) + \ss(\bu^{k-1}(O(b,a,b))).
\end{align*}
This completes the proof since $\bu^{k-1}(O(a,b,a)) = \bu^{k}(T(a,b,a))$.
\end{proof}

In Appendix~\ref{sec:Alt_Proof}, we give an alternative proof of Proposition~\ref{prop:LRzip} without using Corollary~\ref{cor:AS}.

\begin{theorem}
\label{thm:HigherLoop}
For $g\geq 1$ and $k\geq 0$, the ranks of the $\Z/2\Z$-modules $\Ker(\ss|_{\tor\A_{2k+1,k}^c})$ and $\Im(\ss|_{\tor\A_{2k+1,k}^c})$ are respectively greater than or equal to $g(2g-1)$ and $g(2g+1)$.
Furthermore, the ranks are exactly $g(2g-1)$ and $g(2g+1)$ when $k=0,1,2$.
\end{theorem}

\begin{proof}
In this proof, we write $\ss_{2k+1,k}$ for the restriction map $\ss|_{\tor\A_{2k+1,k}^c}$.
It follows from Proposition~\ref{prop:LRzip} that
\[
\bu^{k}(T(a,b,a)) - \bu^{k}(T(b,a,b)) \in \Ker\ss.
\]
Here, by Corollary~\ref{cor:buTaba}, one has $\bu^{k}(T(a,b,a)) \neq 0$.
Comparing labels, we conclude that the above element is non-trivial and $\rank\Ker\ss_{2k+1,k} \geq \binom{2g}{2} = g(2g-1)$.
Also, Corollary~\ref{cor:buTaba} implies that $\rank\Im\ss_{2k+1,k} \geq g(2g-1)+2g = g(2g+1)$.

We finally show that the above inequality is an equality when $k=0,1,2$.
Since we have the exact sequence
\[
0 \to \Ker\ss_{2k+1,k} \to \tor\A_{2k+1,k}^c \to \Im\ss_{2k+1,k} \to 0,
\]
it suffices to see that $\rank\tor\A_{2k+1,k}^c \leq 4g^2$.
By Lemma~\ref{lem:bu_small}, the cases $k=1,2$ are reduced to the case $k=0$.
Then the proof is completed by the fact that $\tor\A_{1,0}^c = \Ang{T(a,b,a) \mid a, b \in \{1^\pm,\dots,g^\pm\}}$.
\end{proof}

\appendix
\section{Alternative proof of Proposition~\ref{prop:LRzip}}
\label{sec:Alt_Proof}
In this appendix, we give an alternative proof of Proposition~\ref{prop:LRzip} by generalizing the proof of \cite[Lemma~6.6(1)]{NSS21}.
The key ingredients are twisted leaves and the zip construction. 
Every box in a clasper $G$ has one output end and two input ends.
To apply the zip construction to $G$, we need to specify a marking on $G$, which is a set of input ends of boxes.
In the following proof, we do not specify markings but simply indicate which boxes are used.
Here we review a useful lemma. 

\begin{lemma}[{\cite[Lemma~A.6]{MaMe13}}]
\label{lem:TwistLemma}
The following equivalences among claspers hold:
\[
\begin{tikzpicture}[scale=0.3, baseline={(0,-0.1)}]
 \draw (-3,0) -- (-1,0);
 \draw (-1,-2) rectangle (0,2);
 \draw (-90:1) arc (-90:90:1);
 \draw (1,0) -- (3,0);
\end{tikzpicture}
\sim
\begin{tikzpicture}[scale=0.3, baseline={(0,0.2)}]
 \draw (-2,0) -- (2,0);
 \draw (0,0) -- (0,2) node[anchor= south] {$\vtwist$};
\end{tikzpicture}
\sim
\begin{tikzpicture}[scale=0.3, baseline={(0,-0.1)}]
 \draw (-3,0) -- (-1,0);
 \draw (90:1) arc (90:270:1);
 \draw (0,-2) rectangle (1,2);
 \draw (1,0) -- (3,0);
\end{tikzpicture}
\quad
\text{and}
\quad
\begin{tikzpicture}[scale=0.3, baseline={(0,0.1)}]
 \draw (0,1) -- (0,2) node[anchor= south] {$\vtwist$};
 \draw (-2,0) rectangle (2,1);
 \draw (-1,0) -- (-1,-1);
 \draw (1,0) -- (1,-1);
\end{tikzpicture}
\sim
\begin{tikzpicture}[scale=0.3, baseline={(0,0.1)}]
 \draw (0,0) rectangle (4,1);
 \draw (5,0) rectangle (9,1);
 \draw (1,1) -- (1,2) node[anchor= south] {$\vtwist$};
 \draw (8,1) -- (8,2) node[anchor= south] {$\vtwist$};
 \draw (3,1) arc (180:0:1.5);
 \draw (2,0) -- (2,-1);
 \draw (7,0) -- (7,-1);
\end{tikzpicture}\ .
\]
\end{lemma}

\begin{proof}[Proof of Proposition~\ref{prop:LRzip}]
Consider the clasper with $k+1$ twisted leaves (also called special leaves)
\[
\begin{tikzpicture}[scale=0.3, baseline={(0,0)}]
 \draw (5,0) -- (0,0) node[draw, circle, anchor=east] {$a$};
 \draw (2,0) -- (2,2) node[anchor= south] {$\vtwist$};
 \draw (4,0) -- (4,2) node[anchor= south] {$\vtwist$};
 \node at (6,0) {$\cdots$};
 \draw (8,0) -- (8,2) node[anchor= south] {$\vtwist$};
 \draw (7,0) -- (10,0) node[draw, circle, anchor=west] {$b$};
\end{tikzpicture}\ .
\]
We first apply Lemma~\ref{lem:TwistLemma} to the leftmost twisted leaf, and then a new box is created.
Applying \cite[Move~11]{Hab00C} to this box $k$ times from left to right, we obtain the equivalent clasper with $2k+1$ boxes
\[
\begin{tikzpicture}[scale=0.3, baseline={(0,0)}]
 \draw (1,1) -- (0,1) node[draw, circle, anchor=east] {$a$};
 \draw (2,2) arc (90:270:1);
 \draw (2,0) -- (9,0);
 \draw (2,1) rectangle (3,5);
 \node at (2.5,3) {\rotatebox{90}{$1$}};
 \draw (3,2) -- (9,2);
 \draw (3,4) -- (4.5,4);
 \draw (4.5,4) arc (180:420:0.5);
 \draw (4.5,4) arc (180:120:0.5);
 \draw (5,0) -- (5,1.5);
 \draw (5,2.5) -- (5,3);
 \draw (5,4) -- (5,6);
 \draw (4,6) rectangle (8,7);
 \node at (6,6.5) {$2$};[xshift=10cm]
 \draw (6,7) -- (6,8) node[anchor= south] {$\vtwist$};
 \draw (7,2) -- (7,6);
 \node at (10,0) {$\cdots$};
 \node at (10,2) {$\cdots$};
 \draw (11,0) -- (12,0);
 \draw (11,2) -- (12,2);
\begin{scope}[xshift=10cm]
 \draw (2,0) -- (9,0);
 \draw (2,1) rectangle (3,5);
 \node at (2.5,3) {\rotatebox{90}{$2k-1$}};
 \draw (3,2) -- (9,2);
 \draw (3,4) -- (4.5,4);
 \draw (4.5,4) arc (180:420:0.5);
 \draw (4.5,4) arc (180:120:0.5);
 \draw (5,0) -- (5,1.5);
 \draw (5,2.5) -- (5,3);
 \draw (5,4) -- (5,6);
 \draw (4,6) rectangle (8,7);
 \node at (6,6.5) {$2k$};
 \draw (6,7) -- (6,8) node[anchor= south] {$\vtwist$};
 \draw (7,2) -- (7,6);
\end{scope}
 \draw (19,-1) rectangle (20,3);
 \draw (20,1) -- (21,1) node[inner sep=3pt, draw, circle, anchor= west] {$b$};
\end{tikzpicture}\ .
\]
By Lemma~\ref{lem:TwistLemma}, this clasper is equivalent to
\[
\begin{tikzpicture}[scale=0.3, baseline={(0,0)}]
 \draw (1,1) -- (0,1) node[draw, circle, anchor=east] {$a$};
 \draw (2,2) arc (90:270:1);
 \draw (2,0) -- (9,0);
 \draw (2,1) rectangle (3,5);
 \node at (2.5,3) {\rotatebox{90}{$1$}};
 \draw (3,2) -- (9,2);
 \draw (3,4) -- (4.5,4);
 \draw (4.5,4) arc (180:420:0.5);
 \draw (4.5,4) arc (180:120:0.5);
 \draw (5,0) -- (5,1.5);
 \draw (5,2.5) -- (5,3);
 \draw (5,4) -- (5,6);
 \draw (4,6) rectangle (8,7);
 \node at (6,6.5) {$2$};[xshift=10cm]
 \draw (6,7) -- (6,8) node[anchor= south] {$\vtwist$};
 \draw (7,2) -- (7,6);
 \node at (10,0) {$\cdots$};
 \node at (10,2) {$\cdots$};
 \draw (11,0) -- (12,0);
 \draw (11,2) -- (12,2);
\begin{scope}[xshift=10cm]
 \draw (2,0) -- (9,0);
 \draw (2,1) rectangle (3,5);
 \node at (2.5,3) {\rotatebox{90}{$2k-1$}};
 \draw (3,2) -- (9,2);
 \draw[rounded corners] (3,4) -- (3.5,4) -- (3.5,10) -- (6.5,10) -- (6.5,8.5);
 \draw (5,0) -- (5,1.5);
 \draw (5,2.5) -- (5,6);
 \draw (4,6) rectangle (6,7);
 \draw (4.5,7) -- (4.5,8) node[anchor= south] {$\vtwist$};
 \draw (8,2) -- (8,6);
 \draw (7,6) rectangle (9,7);
 \node at (10.5,6.5) {\tiny$2k+1$};
 \draw (8.5,7) -- (8.5,8) node[anchor= south] {$\vtwist$};
 \draw (5.5,7) arc (180:140:1);
 \draw (7.5,7) arc (0:100:1);
 \draw (6.5,8.5) arc (90:300:0.5);
 \draw (6.5,8.5) arc (90:30:0.5);
\end{scope}
 \draw (19,-1) rectangle (20,3);
 \draw (20,1) -- (21,1) node[inner sep=3pt, draw, circle, anchor= west] {$b$};
\end{tikzpicture}\ .
\]
Apply the zip construction to the $k+1$ boxes with odd numbers so that one obtains a graph clasper $G$ of degree $2k+1$ arising from the left-top input edge of Box~$2k+1$ after using \cite[Move~2]{Hab00C}.
Hence, by \cite[Lemma~A.1]{MaMe13} (or \cite[Moves~11 and 12]{Hab00C}), $G$ can be separated from the leaf connected to the right-top input edge of Box~$2k-1$ up to $Y_{2k+2}$-equivalence.
We next apply the same zip construction backward, and then we delete the leaf and Box~$2k-1$ by \cite[Move~3]{Hab00C}:
\[
\begin{tikzpicture}[scale=0.3, baseline={(0,0)}]
 \draw (1,1) -- (0,1) node[draw, circle, anchor=east] {$a$};
 \draw (2,2) arc (90:270:1);
 \draw (2,0) -- (9,0);
 \draw (2,1) rectangle (3,5);
 \node at (2.5,3) {\rotatebox{90}{$1$}};
 \draw (3,2) -- (9,2);
 \draw (3,4) -- (4.5,4);
 \draw (4.5,4) arc (180:420:0.5);
 \draw (4.5,4) arc (180:120:0.5);
 \draw (5,0) -- (5,1.5);
 \draw (5,2.5) -- (5,3);
 \draw (5,4) -- (5,6);
 \draw (4,6) rectangle (8,7);
 \node at (6,6.5) {$2$};[xshift=10cm]
 \draw (6,7) -- (6,8) node[anchor= south] {$\vtwist$};
 \draw (7,2) -- (7,6);
 \node at (10,0) {$\cdots$};
 \node at (10,2) {$\cdots$};
 \draw (11,0) -- (12,0);
 \draw (11,2) -- (12,2);
\begin{scope}[xshift=10cm]
 \draw (2,0) -- (9,0);
 \draw (2,2) -- (9,2);
 \draw (5,0) -- (5,1.5);
 \draw (5,2.5) -- (5,6);
 \draw (4,6) rectangle (6,7);
 \draw (4.5,7) -- (4.5,8) node[anchor= south] {$\vtwist$};
 \draw (8,2) -- (8,6);
 \draw (7,6) rectangle (9,7);
 \node at (10.5,6.5) {\tiny$2k+1$};
 \draw (8.5,7) -- (8.5,8) node[anchor= south] {$\vtwist$};
 \draw (5.5,7) arc (180:0:1);
\end{scope}
 \draw (19,-1) rectangle (20,3);
 \draw (20,1) -- (21,1) node[inner sep=3pt, draw, circle, anchor= west] {$b$};
\end{tikzpicture}\ .
\]
In the same manner, one can delete Boxes $2k-3,\dots,3$ and $1$ in this order:
\[
\begin{tikzpicture}[scale=0.3, baseline={(0,0)}]
 \draw (1,1) -- (0,1) node[draw, circle, anchor=east] {$a$};
 \draw (2,2) arc (90:270:1);
 \draw (2,0) -- (9,0);
 \draw (2,2) -- (9,2);
 \draw (5,0) -- (5,1.5);
 \draw (5,2.5) -- (5,6);
 \draw (4,6) rectangle (6,7);
 \draw (4.5,7) -- (4.5,8) node[anchor= south] {$\vtwist$};
 \draw (8,2) -- (8,6);
 \draw (7,6) rectangle (9,7);
 \draw (8.5,7) -- (8.5,8) node[anchor= south] {$\vtwist$};
 \draw (5.5,7) arc (180:0:1);
 \node at (10,0) {$\cdots$};
 \node at (10,2) {$\cdots$};
 \draw (11,0) -- (12,0);
 \draw (11,2) -- (12,2);
\begin{scope}[xshift=10cm]
 \draw (2,0) -- (9,0);
 \draw (2,2) -- (9,2);
 \draw (5,0) -- (5,1.5);
 \draw (5,2.5) -- (5,6);
 \draw (4,6) rectangle (6,7);
 \draw (4.5,7) -- (4.5,8) node[anchor= south] {$\vtwist$};
 \draw (8,2) -- (8,6);
 \draw (7,6) rectangle (9,7);
 \draw (8.5,7) -- (8.5,8) node[anchor= south] {$\vtwist$};
 \draw (5.5,7) arc (180:0:1);
\end{scope}
 \draw (19,-1) rectangle (20,3);
 \draw (20,1) -- (21,1) node[inner sep=3pt, draw, circle, anchor= west] {$b$};
\end{tikzpicture}\ .
\]
Here we choose one of the $2k$ boxes except the rightmost one, and apply \cite[Move~11]{Hab00C} to the box $2k+1$ times so that one obtains a graph clasper of degree $2k+1$ with a twisted leaf, which is deleted by \cite[Lemma~A.5]{MaMe13} up to $Y_{2k+2}$-equivalence.
Eventually, we obtain the clasper with a box
\[
\begin{tikzpicture}[scale=0.3, baseline={(0,0)}]
 \draw (1,1) -- (0,1) node[draw, circle, anchor=east] {$a$};
 \draw (2,2) arc (90:270:1);
 \draw (2,0) -- (9,0);
 \draw (2,2) -- (9,2);
 \draw (5,0) -- (5,1.5);
 \draw (5,2.5) -- (5,4);
 \draw (7,2) -- (7,4);
 \draw (5,4) arc (180:0:1);
 \node at (10,0) {$\cdots$};
 \node at (10,2) {$\cdots$};
 \draw (11,0) -- (12,0);
 \draw (11,2) -- (12,2);
\begin{scope}[xshift=10cm]
 \draw (2,0) -- (9,0);
 \draw (2,2) -- (9,2);
 \draw (5,0) -- (5,1.5);
 \draw (5,2.5) -- (5,4);
 \draw (7,2) -- (7,4);
 \draw (5,4) arc (180:0:1);
\end{scope}
 \draw (19,-1) rectangle (20,3);
 \draw (20,1) -- (21,1) node[inner sep=3pt, draw, circle, anchor= west] {$b$};
\end{tikzpicture}\ .
\]
Here we use \cite[Move~5 or 6]{Hab00C}, and get the right-hand side $\ss(\bu^{k}(T(b,a,b)))$ of the statement.

On the other hand, if we start the above procedure from the rightmost twisted leaf, then we get the left-hand side $\ss(\bu^{k}(T(a,b,a)))$.
\end{proof}

\def\cprime{$'$} \def\cprime{$'$} \def\cprime{$'$}

\end{document}